\documentclass{amsart}

\usepackage{amsthm, amsfonts, amssymb, amsmath, latexsym, enumerate, times}
\usepackage[latin1]{inputenc}
\usepackage{xypic}
\usepackage{mathrsfs}
\usepackage{mathtools}
\usepackage{letltxmacro}
\usepackage{comment}
\LetLtxMacro\orgvdots\vdots
\LetLtxMacro\orgddots\ddots

\makeatletter
\DeclareRobustCommand\vdots{%
	\mathpalette\@vdots{}%
}
\newcommand*{\@vdots}[2]{%
	\sbox0{$#1\cdotp\cdotp\cdotp\m@th$}%
	\sbox2{$#1.\m@th$}%
	\vbox{%
		\dimen@=\wd0 %
		\advance\dimen@ -3\ht2 %
		\kern.5\dimen@
		\dimen@=\wd2 %
		\advance\dimen@ -\ht2 %
		\dimen2=\wd0 %
		\advance\dimen2 -\dimen@
		\vbox to \dimen2{%
			\offinterlineskip
			\copy2 \vfill\copy2 \vfill\copy2 %
		}%
	}%
}
\DeclareRobustCommand\ddots{%
	\mathinner{%
		\mathpalette\@ddots{}%
		\mkern\thinmuskip
	}%
}
\newcommand*{\@ddots}[2]{%
	\sbox0{$#1\cdotp\cdotp\cdotp\m@th$}%
	\sbox2{$#1.\m@th$}%
	\vbox{%
		\dimen@=\wd0 %
		\advance\dimen@ -3\ht2 %
		\kern.5\dimen@
		\dimen@=\wd2 %
		\advance\dimen@ -\ht2 %
		\dimen2=\wd0 %
		\advance\dimen2 -\dimen@
		\vbox to \dimen2{%
			\offinterlineskip
			\hbox{$#1\mathpunct{.}\m@th$}%
			\vfill
			\hbox{$#1\mathpunct{\kern\wd2}\mathpunct{.}\m@th$}%
			\vfill
			\hbox{$#1\mathpunct{\kern\wd2}\mathpunct{\kern\wd2}\mathpunct{.}\m@th$}%
		}%
	}%
}
\makeatother

\usepackage{array}
\usepackage{comment}
\usepackage{paralist}
\usepackage{xcolor}

\newtheorem{theorem}{Theorem}
\newtheorem{lemma}[theorem]{Lemma}
\newtheorem{claim}[theorem]{Claim}
\newtheorem{corollary}[theorem]{Corollary}
\newtheorem{proposition}[theorem]{Proposition}

\theoremstyle{definition}

\newtheorem{remark}[theorem]{Remark}
\newtheorem{example}[theorem]{Example}

\newcommand{\calE}{{\mathcal E}}
\newcommand{\calF}{{\mathcal F}}

\newcommand{\calL}{{\mathcal L}}

\newcommand{\calN}{{\mathcal N}}
\newcommand{\calP}{{\mathcal P}}
\newcommand{\calO}{{\mathcal O}}
\newcommand{\calX}{{\mathcal X}}

\newcommand{\bbC}{{\mathbb C}}
\newcommand{\bbF}{{\mathbb F}}
\newcommand{\bbP}{{\mathbb P}}

\newcommand{\uz}{{\underline{0}}}

\newcommand{\Base}{\operatorname{Base}}
\def\geq{\geqslant}
\def\leq{\leqslant}

\begin{document}
 
\title[Non--cyclic triple planes with branch curve of degree at most 10]
{Non--cyclic triple planes with branch curve of degree at most 10}
 
\author{Ciro Ciliberto}
\address{Dipartimento di Matematica, Universit\`a di Roma Tor Vergata, Via O. Raimondo 00173 Roma, Italia}
\email{cilibert@axp.mat.uniroma2.it}

\author{Rick Miranda}
\address{Department of Mathematics, Colorado State University, Fort Collins (CO), 80523,USA}
\email{rick.miranda@colostate.edu}
 
 
 
 \begin{abstract} 
 	In this paper we classify normal non--cyclic triple covers of $\bbP^2$ 
 	with branch curve of degree at most 10.
 \end{abstract}
 
\maketitle

\tableofcontents

\section*{Introduction} 

Triple coverings of the plane (usually called \emph{triple planes}) 
are a classical object of study in algebraic geometry,
which began to be exposed in earnest in the first half of the 20-th century;
we refer the reader to the papers by
Bronowski \cite {B1,B2,B3}, Du Val \cite {Du1,Du2}, and Pompilj \cite {P0,P01,P1,P,P2}.
In 1985, the foundation of a general theory of triple covers in algebraic geometry 
was set up by the second author of the present paper \cite {M}, 
followed by S.-L. Tan in and Casnati and Ekedahl in \cite {CE}.
Since then multiple authors have written on both theory and examples
of triple covers, and in particular triple planes; see for example \cite {FPV, IPR}.

Cyclic triple covers are straightforward to analyse, and we do not treat them  in this paper.

In the present paper we deal with 
the classification of the non--cyclic triple covers of the plane 
whose branch curve has degree at most $10$. 
This topic has been considered already by P. du Val in \cite {Du1, Du2};
he proposes in the two cited papers 
to classify 
all triple planes with branch curve of degree at most $14$.
However Du Val implicitly makes some generality assumptions that we want to overcome;
for instance he assumes that the branch curve has only cusps as singularities. 
Moreover in his classification there are some missing cases 
(such as, for example, the triple plane with branch curve of degree $12$ 
briefly presented in our Example \ref {ex:ell2}) 
and his discussion of some cases is a bit too informal 
to be acceptable for modern standards of rigor. 
However Du Val's papers contain beautiful ideas and we hope to return to them 
to give a detailed classification of triple planes 
with branch curve of degree $12$ and $14$. 

We should remark that the triple planes that appear in our classification 
also appear in Du Val's papers  \cite {Du1,Du2} 
and some of them also appear in the cited papers \cite {FPV, IPR}, 
although in these two papers the authors make the standard hypothesis 
that the triple plane is smooth, 
whereas we only assume it is normal.

Our classification
of triple planes with branch curve degree at most $10$ 
is complete except for one case (see Section \ref {ssec:opla}), 
concerning the branch curve of degree $10$, 
in which there is a possibility that the triple cover $X$ of the plane 
has $p_g=0$, $q=1$ and either $X$ is properly elliptic 
or $X$ has Kodaira dimension $-\infty$. 
Unfortunately we do not know if these cases can really happen or not. 
It should be mentioned that Du Val asserts in \cite {Du1} that these cases do not occur,
but we cannot completely follow his arguments there.

We now explain in some detail the contents of the paper. 
In Section \ref {sec:prob}, we give the basic definitions and fix notation. 
In Section \ref {sec:pomp} 
we consider triple planes that possess an irrational pencil of curves. 
Here we give, for the reader's convenience, 
the proof of a beautiful theorem of Pompilj (see Theorem \ref {thm:pomp}) 
that asserts that the irrational pencil is \emph{composed} with the triple plane map 
(see the definition in Section \ref {ssec:comp}) 
with some exceptions that are fully described. 
This result turns out to be useful in the rest of the paper. 

In Section \ref {sec:branch} we make some remarks on the branch curve of a triple plane. 
In particular we deal with some cases 
in which there are double components of the branch curve 
and we characterize the case in which the whole branch curve is double 
and the triple plane is not cyclic (see Proposition \ref {prop:simul}). 
It turns out however that if the triple plane is smooth and the branch curve is double, 
then the triple plane is cyclic (see Remark \ref {rem:lpi}).

In Section \ref {sec:4} we consider the case in which the branch curve has degree $4$. 
This is a well known case, treated also in \cite {P1}. 
It turns out that the triple plane is essentially 
the projection of a rational normal scroll  of degree $3$ in $\bbP^4$ 
from a line not intersecting it. 
In general the branch curve is a quartic with three cusps. 
However following Pompilj, we also examine the special case 
in which the branch curve is the union of a cuspidal cubic 
plus its unique flex tangent line.  
In Section \ref {ssec:rot} we examine this example 
from the viewpoint of the structure constants of the algebra,
in the spirit of \cite{M}. 

Section \ref {sec:6} is devoted to the study of the case 
in which the branch curve has degree $6$. 
In this case the triple plane is essentially the projection 
of a normal cubic surface $X$ in $\bbP^3$ to a plane from a point not on it. 
If $X$ is smooth and the projection is general, 
the branch curve is an irreducible sextic that has in general six ordinary cusps. 
However we also examine the special case in which the branch curve is reducible, 
consisting of two smooth cubics 
that intersect in three collinear points with intersection multiplicity $3$ at each of them. 
In Section \ref {ssec:iol} we examine this example 
from the viewpoint of the structure constants of the algebra. 

In Section \ref {sec:8} we deal with the branch curve of degree $8$. 
Here the irregularity of the triple cover is a priori $q\leq 2$. 
However we prove in Proposition \ref {prop:deg=8notq=1} 
that the case $q=1$ does not occur. 
The proof of this fact is rather delicate and takes several pages. 
In the case $q=0$ the triple plane is rational 
and is fully described in Proposition \ref {prop:plane}. 

Finally in Section \ref {sec:10} we consider the case of the branch curve of degree $10$. 
The cases that occur are essentially the following: 
the projection to a plane of a quartic K3 surface in $\bbP^3$ from a smooth point on it,
and four rational cases that are fully described in Proposition \ref {prop:plop2} 
(see also Remark \ref {rem:skip} and Example \ref {ex:gul}). 
As we said, there is the possibility that in this case 
the triple plane has $p_g=0$, $q=1$ 
and either $X$ is properly elliptic or $X$ has Kodaira dimension $-\infty$, 
but we do not know if these cases really occur or not. 

In the brief Section \ref {sec:12}
we give some examples of non--cyclic triple planes with branch curve of degree $d>10$. 
In the Appendix \ref {triplecoverformulae} we review 
how to compute the structure constants of the algebra
(in the spirit of \cite{M})
when the triple plane is defined by a monic cubic polynomial.\bigskip

{\bf Aknowledgements:} The first author is a member of GNSAGA of the Istituto Nazionale di Alta Matematica ``F. Severi''.

\section{Preliminaries and notation}\label{sec:prob}

\subsection{General facts} 
For us a \emph{ triple plane} 
will be the datum of an irreducible, normal, projective complex surface $X$ 
and a finite, degree $3$, morphism $\pi: X\longrightarrow \bbP^2$. 
We will say that the triple plane is \emph{ordinary} if it is not cyclic.
Often we will abbreviate this and say that $X$ is an ordinary triple plane (otp).

Given a $\pi: X\longrightarrow \bbP^2$, 
we consider the $2$--dimensional linear system 
$\calL$ of Cartier divisors that is the pull--back via $\pi$ 
of the linear system of the lines of $\bbP^2$. 
The linear system $\calL$ is base point free, 
the general curve $C\in \calL$ is smooth and irreducible, 
$C^2=3$ and for any curve $D$ on $X$ one has $C\cdot D>0$;
hence $C$ is ample.  
The linear series cut out by $\calL$ on the general curve $C\in \calL$ 
is a base point free $g^1_3$. 

Let $\phi: \tilde X\longrightarrow X$ be the minimal desingularization of $X$ 
and let $\tilde {\calL}$ be the pull back via $\phi$ of $\calL$. 
We will abuse notation and denote by $C$ also the general curve in $\tilde {\calL}$. The curve $C$ is big and nef on $\tilde X$. 
By the minimality assumption on $\phi$, 
for any $(-1)$--curve $\theta$ on $\tilde X$,
one has $C\cdot \theta>0$.

Given a triple plane $\pi: X\longrightarrow \bbP^2$, 
one has $\pi(\calO_X)=\calO_{\bbP^2}\oplus E$ with $E$ a rank 2 vector bundle, 
called the \emph{Tschirnhausen vector bundle} of the triple plane (see \cite {M}). 

We will denote by $B\subset \bbP^2$ the branch curve of the otp 
and by $R\subset X$ the ramification curve. 
The branch curve $B$ can be non--reduced, 
but, under our assumption that $X$ is normal, 
$B$ can have at most double components (over which there is total ramification);
it can even happen that all components of $B$ appear with multiplicity $2$ 
and the triple plane is not cyclic (see  Section \ref {sec:branch}).
The degree $d=2h$ of $B$ is even, 
and the (arithmetic) genus of the curves in $\calL$ is $g=h-2$. 
Hence $h\geq 2$ and therefore $d\geq 4$. 

Given two triple planes $\pi: X\longrightarrow \bbP^2$ and $\pi': X'\longrightarrow \bbP^2$ 
we will say that they are \emph{birationally equivalent} or \emph{Cremona equivalent} 
if there is a commutative diagram of the form
\begin{equation}\label{eq:loh}
\xymatrix{
X' \ar[d]_{\pi'}\ar@{-->}[r] &
X\ar[d]^\pi\\
\bbP^2 \ar@{-->}[r]& \bbP^2   }
\end{equation}
such that $X'\dasharrow X$ and $\bbP^2\dasharrow \bbP^2$ are birational maps. 

Let $X$ be the plane blown up at $n$ proper or infinitely near points. 
We let $E_1,\ldots, E_n$ be the $(-1)$--cycles corresponding to the blown--up points 
and $H$ the pull back on $X$ of a general line in the plane. 
Any complete linear system on $X$ is of the form $|dH-\sum_{i=1}^n m_iE_i|$, 
where $d$ is called the \emph{degre} of the linear system 
and $m_1,\ldots, m_n$ are called the \emph{multiplicities}. 
We will often denote such a linear system with the notation $|d; m_1,\ldots, m_n|$ (exponential notation being used for repeated multiplicities).

\subsection{Triple cover data}\label{ssec:tcd} 
Let $Z$ be a smooth projective irreducible complex variety, 
and let $\pi: X\longrightarrow Z$ with $X$ normal, 
and  $\pi$ a flat finite, degree $3$ morphism. 

We recall from \cite [Sect. 1]{T} that a triple  $(\calN, s, t)$ 
is called a \emph{triple cover datum} on the variety $Z$ 
if $\calN$ is an invertible sheaf on $Z$ 
and $s$ and $t$ are global sections 
of $\calN^{\otimes 2}$ and of $\calN^{\otimes 3}$ respectively, 
with $t\neq 0$. 

A triple cover datum $(\calN, s, t)$ on $Z$ determines a triple cover $\pi: X\longrightarrow Z$, with $X$ normal, in the following way  (see \cite [Sect. 1.1] {T}).
Let $V (\calN)$ be the associated line bundle of $\calN$ with the vector bundle map $p: V (\calN)\longrightarrow Z$. Let $z$ be the global coordinate in the fibers of $V (\calN)$. 
Then $z$ is a global section of $p^*(\calN)$. 
Thus we obtain the polynomial section of $p^*(\calN^{\otimes 3})$ given by 
$$f(z):=z^3 +sz+t,$$
where $s$ and $t$ are viewed as sections of $p^*(\calN^{\otimes 2})$ and of $p^*(\calN^{\otimes 3})$. Then $f(z)=0$ defines a codimension 1  subscheme $S$ of $V (\calN)$. Let $X$ be the normalization of $S$.  Then the composition of the normalization map with the bundle projection $p$ defines a map $\pi: X\longrightarrow Z$ with $X$ normal, and  $\pi$ a flat finite, degree $3$ morphism. 

Conversely,  every triple cover $\pi: X\longrightarrow Z$ 
can be constructed by this method (see \cite [Sect. 7] {T}),
although the triple $(\calN, s, t)$ is not uniquely determined by $\pi$.

\section{Triple planes with an irrational pencil}\label{sec:pomp}

In this section we will recall an interesting theorem due to G. Pompilj (see \cite{P});
for the reader's convenience, 
we will give a proof that follows quite closely Pompilj's original proof. 

An irreducible, normal, projective surface $X$ 
is said to have an \emph{irrational pencil} of genus $\gamma>0$ 
if there is a dominant rational map $f: X\dasharrow \Gamma$, 
where $\Gamma$ is a smooth curve of genus $\gamma$ 
and the general fibre $F$ of $f$ is irreducible. 
We will denote by $p$ the geometric genus of the general fibre $F$ of $f$. 
Notice that $f: X\dasharrow \Gamma$ is a morphism if $X$ is smooth 
or it has only rational singularities. 

Pompilj's theorem (see Theorem \ref{thm:pomp} below) 
is a characterization of triple planes possessing an irrational pencil. 

\subsection{The exceptional examples}\label{ssec:exc}  
Let $\Gamma$ be an irreducible, smooth, projective curve of genus $1$, 
that we can realize as a smooth cubic curve in $\bbP^2$. 
Consider $\Gamma[2]$ the 2--fold symmetric product of $\Gamma$. 
The points of $\Gamma[2]$ are effective divisors $x+y$ of degree $2$ on $\Gamma$.

There is an obvious surjective morphism $f: \Gamma[2]\longrightarrow \Gamma$ 
such that $f(x+y)=x\oplus y$, where $\oplus$ is the addition on $\Gamma$ 
(which is a $1$--dimensional torus). 
The fibres of $f$ are smooth and rational, 
and so $\Gamma[2]$ has an irrational pencil;
it is a minimal elliptic ruled surface.

There is a finite, degree $3$ morphism 
$\pi: \Gamma[2]\longrightarrow \check\bbP^2 (\cong \bbP^2)$ , 
which realizes $\Gamma[2]$ as a triple plane, defined as follows. 
Consider $\Gamma$ as a smooth cubic in $\bbP^2$. 
Let $x+y\in \Gamma[2]$. 
Then there is a unique  line $r_{x+y}$ that cuts out on $\Gamma$ 
a divisor containing $x+y$. 
The map $\pi$ sends $x+y$ to $r_{x+y}\in \check \bbP^2$. 
It is clear that $\pi$ has degree $3$ 
and the branch curve of $\pi: \Gamma[2]\longrightarrow \bbP^2$ 
is the dual of $\Gamma$ which is a sextic curve with $9$ cusps as singularities. 
(These $9$ cusps correspond to the flex tangents of $\Gamma$.)
It turns out that the Tschirnhausen vector bundle of 
$\pi: \Gamma[2]\longrightarrow \bbP^2$ is $\Omega^1_{\bbP^2}$ 
(see \cite [p. 1158] {M}). 
The triple plane $\pi: \Gamma[2]\longrightarrow \bbP^2$ is ordinary 
and will be called the \emph{basic exceptional example}. 

Now consider a dominant rational map $\phi: \bbP^2 \dasharrow \bbP^2$. 
Let $\pi': X\longrightarrow \bbP^2$ be such that the following diagram is cartesian:\begin{equation*}\label{eq:ex}
\xymatrix{
X \ar[d]_{\pi'}\ar@{-->}[r] &
\Gamma[2]\ar[d]^{\pi}\\
\bbP^2 \ar@{-->}[r]^{\phi}& \bbP^2   }
\end{equation*}
Then, by substituting $X$ with its normalization if necessary,
$\pi': X\longrightarrow \bbP^2$ is also an otp 
and $X$ has an irrational pencil of genus 1. 

We will call any otp of the type $\pi': X\longrightarrow \bbP^2$ as above 
an \emph{exceptional example}. 

\subsection{ Pencil composed with a triple plane map}\label{ssec:comp}   
Let $\pi: X\longrightarrow \bbP^2$ be a triple plane   
with a pencil $f: X\dasharrow \Gamma$ of genus $\gamma\geq 0$ 
of curves of genus $p$. 
We will say that the  pencil is \emph{composed} with the triple plane map 
if there is a commutative diagram
\begin{equation} \label{eq:comp}
\xymatrix{
X \ar[d]_{\pi}\ar@{-->}[r]^f &
\Gamma\ar[d]^\varphi\\
\bbP^2 \ar@{-->}[r]^h& \bbP^1   }
\end{equation}
where $\varphi: \Gamma\longrightarrow \bbP^1$ is a degree $3$ morphism, 
and $h: \bbP^2\dasharrow \bbP^1$ is a rational map 
determined by a fixed part free pencil $\calP$ 
of curves of degree $d$ and genus $p$. 
In this case the pull--back via $\pi$ of a general curve $D\in \calP$ 
consists of the union of three curves of the pencil $f: X\dasharrow \Gamma$.
Moreover the branch curve $B$ of $\pi: X\longrightarrow \bbP^2$ 
is contained in the pull--back via $h$ 
of the branch divisor of $\varphi: \Gamma\longrightarrow \bbP^1$;
hence it consists of a union of irreducible components 
contained in $2\gamma+4$ curves of $\calP$
and therefore it has degree at most $(2\gamma+4)d$. 
If $\calP$ is a pencil of rational curves, 
there are no multiple curves in $\calP$, 
so every fibre of $h$ over a branch point of $\varphi: \Gamma\longrightarrow \bbP^1$ contributes some component to the branch curve; 
hence in this case the degree of the branch curve is at least $2\gamma+4$.

\subsection{Pompilj's theorem}
 Pompilj's theorem mentioned above is the following:

\begin{theorem}\label{thm:pomp} 
Let $\pi: X\longrightarrow \bbP^2$ be a triple plane 
with an irrational pencil $f: X\dasharrow \Gamma$ of genus $\gamma$ 
of curves of genus $p$. 
Then either $\pi: X\longrightarrow \bbP^2$ is Cremona equivalent 
to an exceptional example 
or the irrational pencil $f: X\dasharrow \Gamma$ is composed with the triple plane map, hence the branch curve $B$ is contained in  $2\gamma+4$ curves of a pencil. 
\end{theorem}

\begin{proof} 
Let $F$ be a general curve of the irrational pencil $f: X\dasharrow \Gamma$ 
and let $x\in F$ be a general point of $F$. 
Hence $x$ is a general point of $X$. 
Set  $\pi^{-1}(\pi(x))=\{x,y,z\}$, 
with $x\neq y\neq	 z\neq x$. 
A priori we have three possibilities:\\
\begin{inparaenum}
\item [(a)] $y$ and $z$ do not belong to $F$;\\
\item [(b)] only one among $y$ and $z$ belongs to $F$;\\
\item [(c)] $y$ and $z$ both belong to $F$.
\end{inparaenum}

\begin{claim}\label{cl:one} 
Cases (b) and (c) cannot happen. 
\end{claim}

\begin{proof}[Proof of Claim \ref {cl:one}] 
Suppose case (b) happens and suppose that $y\in F$. 
Then, since $x$ is a general point of $X$, 
we would have an involution $\iota: F\dasharrow F$ 
that maps the general point $x\in F$ to $y\in F$. 
Let $g: X\dasharrow Y$ be the quotient of $X$ by the involution $\iota$. 
Then we would have a commutative diagram of the form
$$
\xymatrix{
X \ar[d]_{\pi}\ar@{-->}[r]^g &
Y\ar@{-->}[d]^\varphi\\
\bbP^2 \ar[r]^{\rm id}& \bbP^2   }
$$
where $\varphi$ is the map that sends the pair $[x,y]$ to $\pi(x)=\pi(y)$. 
Since $\pi$ has degree $3$ and $g$ has degree $2$, this is not possible. 

Suppose that case (c) occurs. 
Then $\pi$ maps the general curve $F$ of the pencil 3:1 to a curve $F'$ of $\bbP^2$, 
and the curve $F'$ would move in a pencil 
in the plane parametrized by the curve $\Gamma$. 
Since $\bbP^2$ has no irrational pencils, $\Gamma$ would be rational, 
hence $\gamma=0$, a contradiction.
\end{proof}

So only case (a) can occur. There are still three cases to be considered:\\
\begin{inparaenum}
\item [(a1)] when $x$ moves on $F$, 
neither $y$ nor $z$ move on a curve of the pencil $f: X\dasharrow \Gamma$ 
different from $F$;\\
\item [(a2)] when $x$ moves on $F$, 
only one among $y$ and $z$ moves on a curve of the pencil different from $F$;\\
\item [(a3)] when $x$ moves on $F$, 
both $y$ and $z$ move on two distinct curves of the pencil different from $F$.
\end{inparaenum}

\begin{claim}\label{cl:two} Case (a2) cannot happen. \end{claim}

\begin{proof}[Proof of Claim \ref {cl:two}] 
The proof is similar to the proof in Claim \ref {cl:one} that case (b) does not occur, 
and can be left to the reader. \end{proof}

In case (a3) the irrational pencil $f: X\dasharrow \Gamma$ 
is composed with the triple plane map, 
so we are done in this case. It remains to discuss case (a1). 

\begin{claim}\label{cl:three} 
In case (a1) one has $\gamma=1$. 
\end{claim}

\begin{proof}[Proof of Claim \ref {cl:three}] 
Consider the map $\phi: \bbP^2\dasharrow \Gamma[3]$ 
that is defined in the following way. 
Let $x\in \bbP^2$ be a general point. 
Set $\pi^{-1}(x)=\{x_1,x_2,x_3\}$, with $x_1,x_2,x_3$ pairwise distinct. 
Then define $\phi(x):=f(x_1)+f(x_2)+f(x_3)\in \Gamma[3]$. 
Because of the assumption (a1), the map $\phi$ is generically finite, 
hence its image $G$ has dimension $2$, and it is rational, being unirational. 
But then $G$ is contained in a complete $g_3^r$, with $3\geq r\geq 2$. 
However, $r=3$ is not possible, because then $\Gamma$ would be rational. 
So the only possibility is $r=2$ and this yields that 
$\Gamma$ has genus $\gamma=1$, as claimed.
\end{proof}

Next we consider the 1--dimensional family $\calF$ 
of the images on $\bbP^2$ of the curves $F$ of the pencil $f: X\dasharrow \Gamma$. 
We denote by $d$ the degree of these curves 
so that $\calF\subset |\calO_{\bbP^2}(d)|$ 
and note that the curves in $\calF$ have geometric genus $p$. 
Note that $\calF$ is a curve of geometric genus $\gamma=1$ in $|\calO_{\bbP^2}(d)|$. 
Moreover we notice that given a general point $x\in \bbP^2$, 
there are exactly three curves of the family $\calF$ passing through $x$.  
This implies that $\calF$ has degree $3$ in $|\calO_{\bbP^2}(d)|$, 
so it is a smooth plane cubic. 
Let $\calP$ be the plane spanned by $\calF$ in $|\calO_{\bbP^2}(d)|$. 
So $\calP$ is a net of degree $d$ curves, that has no fixed components, 
because $\calF$ has no fixed components. 
Hence $\calP$ has finitely many base points (if any) 
and we can denote by $m$ the number of intersection points 
of two general curves in $\calP$ off the base points (counted with multiplicities). 

\begin{claim}\label{cl:four} 
If $m=1$ then $\pi: X\longrightarrow \bbP^2$ is Cremona equivalent 
to a basic exceptional example. 
\end{claim}

\begin{proof}[Proof of Claim \ref {cl:four}] 
If $m=1$, then $\calP$ is a homaloidal net of curves of genus $0$, 
so that $p=0$ and $X$ is an elliptic ruled surface. 
Up to a Cremona transformation, $\calP$ can be transformed to 
the net of lines $|\calO_{\bbP^2}(1)|\cong \check \bbP^2$ of $\mathbb P^2$ 
and $\calF$ is transformed to a smooth cubic in $\check \bbP^2$. 
Then the branch curve $B$ of the triple plane is the dual of this smooth cubic curve, 
i.e., it is a sextic curve with $9$ cusps as singularities.  
The tangent lines to $B$ are the images of the fibres of the irrational pencil of $X$. 
Let $\phi: \tilde X\longrightarrow X$ be as usual the minimal desingularization of $X$. 
If $r$ is a general tangent line to $B$, 
the pull--back of $r$ to $\tilde X$ is of the form $F+\Gamma$, 
where $F$ is the general fibre of the irrational pencil of $\tilde X$, 
and $\Gamma$ is a smooth elliptic curve. 
We have  $3=(\Gamma+F)^2=\Gamma^2+2$, so $\Gamma^2=1$. 
Then by applying \cite  [Prop. (0.18)] {CCML} 
we see that $\tilde X$ is birational to $\Gamma[2]$, and this proves the assertion.
\end{proof}

The following claim finishes the proof of Pompilj's theorem:

\begin{claim}\label{cl:five} 
If $m>1$ then $\pi: X\longrightarrow \bbP^2$ 
is Cremona equivalent to an exceptional example. 
\end{claim}

\begin{proof}[Proof of Claim \ref {cl:five}] 
Suppose $m>1$. 
Consider the degree $m$ rational map $\phi_\calP: \bbP^2\dasharrow \bbP^2$ 
defined by the net $\calP$.  
Then $\phi_\calP$ maps the curves in $\calP$ to  $|\calO_{\bbP^2}(1)|$ 
and the family $\calF$ is mapped to a family $\calF'$ of lines of $\bbP^2$, 
such that, given a general point $x$ of $\bbP^2$,
there are exactly three curves in $\calF'$ containing $x$. 
Consider the universal family $X'\longrightarrow \calF'$. 
It has an obvious map $\pi': X'\longrightarrow \bbP^2$ 
which  is clearly Cremona equivalent  
to the basic exceptional example. 
Moreover we have a commutative diagram 
\begin{displaymath}
\xymatrix{
X \ar[d]_{\pi}\ar[r] &
X'\ar[d]^{\pi'}\\
\bbP^2 \ar@{-->}[r]^{\phi_\calP}& \bbP^2   }
\end{displaymath}
and this implies the assertion. 
\end{proof}
As mentioned above, this claim finishes the proof of Pompilj's theorem.
\end{proof}

\begin{corollary}\label{cor:pomp} 
Let $\pi: X\longrightarrow \bbP^2$ be a triple plane  
such that $X$ has Kodaira dimension $-\infty$ and $X$ is of irregularity $q>0$. 
Then either the triple plane is Cremona equivalent to an exceptional example, 
or $\pi: X\longrightarrow \bbP^2$ is Cremona equivalent 
(as in diagram \eqref {eq:loh}) 
to a  triple plane $\pi': X'\longrightarrow \bbP^2$ 
such that the branch curve consists of the union of $2q+4$ lines in a pencil. 
In this case the degree of the branch curve of $\pi: X\longrightarrow \bbP^2$ 
is at least $2q+4$ and the equality occurs if and only if 
the branch curve of $\pi: X\longrightarrow \bbP^2$ consists of $2q+4$ lines in a pencil. 
\end{corollary} 

\begin{proof} 
Since $X$ has Kodaira dimension $-\infty$ and has irregularity $q\geq 1$, 
then $X$ has an irrational pencil $f: X\dasharrow \Gamma$ of genus $q$ 
of curves of genus $0$. 

If the given triple plane is not an exceptional example, then, 
by Pompilj's Theorem \ref{thm:pomp}, 
the irrational pencil $f: X\dasharrow \Gamma$ is composed with the triple plane map. 
This implies that there is a diagram as \eqref {eq:comp} 
where $h$ is determined by a pencil of rational curves. 
On the other hand any pencil of rational curves on $\bbP^2$ 
is Cremona equivalent to a pencil of lines (see \cite [Prop. 5.3.5] {C}). 
So we have a diagram \eqref {eq:loh} 
such that the branch curve of $\pi': X'\longrightarrow \bbP^2$  
consists of the union of $2q+4$ lines in a pencil. 
The final assertion follows from the content of Section \ref {ssec:comp}.
\end{proof}

\section{Remarks on the branch curve of a triple plane}\label{sec:branch}

By the considerations of Section \ref {ssec:tcd}, 
we can think of a triple plane $\pi: X\longrightarrow \mathbb P^2$ 
as being given by an equation of the type
\begin{equation}\label{eq:slap}
f(x,y,t,z):=z^3+3az+b=0,
\end{equation}
where $a=a(x,y,t)$ and $b=b(x,y,t)$ are homogenous polynomials 
of degrees $2d$ and $3d$ respectively 
and $[x,y,t]$ are the homogenous coordinates in the plane. 
In this case the branch curve $B$ has equation 
$$
4a^3+b^2=0
$$
and has degree $6d$. 
We may and will assume that $B$ has no triple component. 
The triple plane is cyclic if and only if $a=0$.  

\begin{lemma}\label{lem:kop} 
In the above set--up, 
if $a$ and $b$ have a  common factor $c$ of positive degree 
then the non--reduced curve with equation $c^2=0$ 
appears as a part of the branch curve of the triple plane 
$\pi: X\longrightarrow \mathbb P^2$. 
In this case $X$ is singular. 
\end{lemma}

\begin{proof} The first assertion is trivial. As for the second assertion, set
$$
a=c \alpha, \quad b=c \beta.
$$
We have
\begin{align*} \label{triplecoverstructureconstants}
	\frac {\partial f} {\partial z} &= 3z^2+3c\alpha\\
	\frac {\partial f} {\partial x} &= 3z \frac {\partial a} {\partial x}+\frac {\partial c} {\partial x}\beta+ c \frac {\partial \beta} {\partial x}\\
	\frac {\partial f} {\partial y} &= 3z \frac {\partial a} {\partial y}+\frac {\partial c} {\partial y}\beta+ c \frac {\partial \beta} {\partial y}	\\
	\frac {\partial f} {\partial t} &= 3z \frac {\partial a} {\partial t}+\frac {\partial c} {\partial t}\beta+ c \frac {\partial \beta} {\partial t}
\end{align*}
and we see that all derivatives vanish 
for $z=c=\beta = 0$, 
and also for all points (if any) of $z=c=0$ that are singular for $c=0$. 
\end{proof} 

Next we consider the case in which $a$ and $b$ have no non--trivial common factor, 
and ask when does it occur that all components of the branch curve $B$ 
appear in $B$ with multiplicity $2$. 
This is certainly the case if the triple cover is cyclic, 
because in this case $a=0$ and the branch curve is $b^2=0$. 
There are however other cases in which this can happen. First we make an example.

\begin{example}\label{ex:trip} 
Let us refer back to Section \ref {ssec:comp} 
and let $\pi: X\longrightarrow \bbP^2$ be a triple plane 
with a pencil $f: X\dasharrow \Gamma$ of genus $\gamma\geq 0$ 
of curves of genus $p$ that is composed with the triple plane map. 
Looking at \eqref {eq:comp}, 
suppose that the branch divisor of $\varphi: \Gamma \longrightarrow \mathbb P^1$ 
consists of $\gamma+2$ distinct points each counted with multiplicity $2$. 
Then the branch curve of the triple plane consists of $\gamma+2$ curves of a pencil 
each counted with multiplicity $2$. 

In general such a triple plane is not cyclic. 
Indeed, consider two non--zero homogenous polynomials $h,k$ of degree $d$, 
with no non--trivial common factor. 
Consider the pencil $\calP$ with equation $\lambda h^3+\mu k^3=0$, 
that consists of (some) triples of curves of the pencil with equation $\lambda h+\mu k=0$. Set $b:=2(h^3-k^3)=0$ and $g:=2(h^3+k^3)=0$. 
It is readily seen that in the pencil generated by $b^2=0$ and $g^2=0$ 
there is also the curve with equation $(hk)^3=0$. 
Now set $a:=4hk$. 
Then we have $4a^3+b^2=g^2$. 
So the triple plane with equation \eqref {eq:slap} 
has the branch curve with equation $g^2=0$ but is not cyclic, 
because $a\neq 0$. 
 
An explicit example of this type is the following. 
Consider the cubic surface defined by the equation
\[
z^3 -(L^2+M^2+LM)z - LM(L+M) = 0
\]
where $L$ and $M$ are homogeneous linear polynomials in $x,y,t$.
The projection to the $[x:y:t]$ plane exhibits this cubic as a non-cyclic triple plane.
The branch curve is calculated to be
\[
(L-M)^2(L+2M)^2(2L+M)^2 = 0
\]
which consists of three double lines in a pencil.
\end{example}

\begin{proposition}\label{prop:simul} 
Consider a triple plane $\pi: X\longrightarrow \mathbb P^2$ 
given by an equation as \eqref{eq:slap}, 
such that the homogeneous polynomials $a$ and $b$ are both non--zero 
(of degrees $2d$ and $3d$ respectively) 
have no non--trivial common factor 
and such that all components of the branch curve $B$ appear in $B$ with multiplicity $2$, 
so that one has
$$
4a^3+b^2=g^2.
$$
Then $a,b,g$ are as in Example \ref {ex:trip}, 
namely there are two non--zero homogenous polynomials $h,k$ of degree $d$, 
with no non--trivial common factor, 
such that   $b=2(h^3+k^3)$ and $g=2(h^3-k^3)$ and  $a=4hk$. 
 \end{proposition}
 
 \begin{proof} 
Note that $b$ and $g$ have no non--trivial common factor, 
otherwise $a$ and $b$ would have a non--trivial common factor, 
contrary to the hypotheses. 
Moreover we have $4a^3=g^2-b^2$, 
which means that the curve $a^3=0$ 
belongs to the fixed component free pencil $\calP$ 
generated by the curves $b^2=0$ and $g^2=0$. 
Consider the fixed component free pencil $\calP'$ 
generated by the curves $b=0$ and $g=0$. 
Then the curves in $\calP$ consist of pairs of curves of $\calP'$, 
and precisely $\calP$ is the $g^1_2$ on $\calP'$ 
whose branch points are the curves $b=0$ and $g=0$. 
Since the curve $a^3=0$ belongs to $\calP$, 
we can write $a^3=4uv$, 
where $u=0$, $v=0$ are curves of $\calP'$ 
and $u,v$ have no non--trivial common factor. 
This implies that there are two non--zero homogenous polynomials $h,k$ of degree $d$, 
with no non--trivial common factor, 
such that $u=h^3, v=k^3$, so that $a=4hk$. 
Now we have $4uv=4(hk)^3=g^2-b^2=(g+b)(g-b)$, 
with $g+b,g-b$ having no non--trivial common factor. 
Since $u=0$, $v=0$, $g+b=0$ and $g-b=0$ are curves of $\calP'$, 
we may assume that $g+b=u=h^3$ and $g-b=v=k^3$, 
hence $g=2(h^3+k^3)$, $b=2(h^3-k^3)$;
this is then in the form of Example \ref{ex:trip} as desired. 
\end{proof}
 
\begin{remark}\label{rem:lpi} 
Consider a triple plane $\pi: X\longrightarrow \bbP^2$ as in Proposition \ref {prop:simul}. Then $X$ is singular.
Indeed we have $b=2(h^3+k^3)$ and $a=4hk$, 
with $h,k$ with no non-trivial common factor, 
and over the points defined by $h=k=0$ the surface $X$ is singular 
by \cite[Lemma 5.1, (a)]{M}. 

Therefore a triple plane $\pi: X\longrightarrow \bbP^2$ with $X$ smooth 
and such that each component of the branch curve $B$ appears in $B$ with multiplicity $2$, is necessarily cyclic. 
\end{remark}

\section{Branch curve of degree 4} \label{sec:4}

 \subsection{The general discussion}  
Consider an otp  $\pi: X\longrightarrow \bbP^2$ such that its branch curve has degree $4$. 
Then the general curve $C\in \calL$ is smooth and rational with $C^2=3$. 
This implies that $\dim(|C|)=4$ 
and $\varphi_{|C|}: X\longrightarrow \Sigma \subset \bbP^4$ 
embeds $X$ as a rational normal scroll $\Sigma\cong X$ of degree $3$ in $\bbP^4$.
Moreover $\calL\subset |C|$ and therefore $\pi: X\longrightarrow \bbP^2$  
can be identified with a projection of $\Sigma$ to $\bbP^2$ 
from a line $\ell\subset \bbP^4$ such that $\ell\cap \Sigma=\emptyset$. 
We have two possibilities: (a) $\Sigma$ is a cone; (b) $\Sigma$ is smooth. 

In the former case the branch curve $B$ of the triple plane 
 consists of four lines belonging to a pencil
(and some lines can occur with multiplicity $2$, see Example \ref {ex:trip}). 
In the latter case  we have the following result, already due to Pompilj in \cite {P1}:

\begin{proposition}\label{prop:pomp} 
If $X=\Sigma$ is smooth, then the quartic branch curve can either be irreducible, 
with three distinct cusps as singularities, 
or reducible in a cuspidal cubic curve plus its unique flex tangent line. 
\end{proposition}  

\begin{proof}
First of all recall that $\Sigma\cong \mathbb F_1$;
its Picard group is generated by a fibre $F$ 
of the structure map $\bbF_1\longrightarrow \bbP^1$ 
and by the section $E$ such that $E^2=-1$. 
The only lines on $\Sigma$ are the curves in the pencil $|F|$ and $E$. 
Moreover the hyperplane bundle $H$ on $\Sigma$ is such that $H\sim E+2F$.

As we know, the triple plane map $\pi: X\longrightarrow \bbP^2$  
can be identified with a projection of $\Sigma$ to $\bbP^2$ 
from a line $\ell\subset \bbP^4$ 
such that $\ell\cap \Sigma=\emptyset$. 
Let $x\in \ell$ be a point and first project $\Sigma$ down to $\bbP^3$ from $x$. 
The image of $\Sigma$ is a degree $3$ surface $\Sigma'$ 
that has a line $r$ of double points.
The image of $\ell$ via the projection is a point $y\in \bbP^3$ not lying on $\Sigma'$, 
and we have to project $\Sigma'$ down to $\bbP^2$ from $y$. 
Let $Q$ be the polar quadric of $\Sigma'$ with respect to $y$. 
Then $Q$ contains the line $r$ and further intersects $\Sigma'$ 
in a quartic curve $D$ whose projection from $y$ is the branch curve $B$ of the triple plane. 

Suppose first that the branch curve $B$ is irreducible.  
Then the curve $D$ is also irreducible 
and it is rational because it intersects the general line of the ruling of $\Sigma'$ in one point. If $p\in D$ is any point, 
then the line $\langle y,p\rangle$ has intersection multiplicity
at least $2$ with $\Sigma$ at $p$. 
This proves that there are no proper secant lines to $D$ passing through $y$. 
Hence the projection of $D$ from $y$ is a rational quartic curve 
with only cuspidal singularitities. 
On the other hand an irreducible quartic curve can have only simple cusps. 
This proves that $B$ has exactly three cusps. 

Suppose next that the branch curve $B$ is reducible. 
First we claim that $B$ cannot contain an irreducible conic. 
Suppose by contradiction that $B$ contains an irreducible conic $A$. 
Then the pull--back of $A$ on $\Sigma$ would belong to $|\calO_\Sigma(2)|$. 
On the other hand this pull--back would consist of 
an irreducible conic $A_0$ counted twice plus another irreducible conic $A_1$. 
But the irreducible conics on $\Sigma$ are all linearly equivalent 
and therefore we would find that $2A_0+A_1\sim 3A_0\in |\calO_\Sigma(2)|$, 
which is not possible. 

So, if $B$ is reducible, then $B$ contains a line $a$. 
We claim that $\pi^*(a)=E+2F_0$, with $F_0\in |F|$. 
In fact since $\pi$ is branched over $a$,
$\pi^*(a)\in |E+2F|$ must be of the form $\pi^*(a)=2A_0+A_1$, 
where $A_0,A_1$ are lines on $\Sigma$. 
It is clear that $A_0\neq E$ 
because $2E+F$ only moves in a pencil on $\Sigma$,
while $\pi^*(a)$ moves in a net at least.
So $A_0$ is a curve $F_0\in |F|$ and therefore $A_1=E$
since the self-intersection must be $3$.

This implies that $a$ is the unique line in $B$. 
Indeed $a$ is uniquely determined by being the image of $E$ 
under the projection $\pi: \Sigma\longrightarrow \bbP^2$. 
Thus $B$ is reducible in $a$ plus an irreducible cubic $Z$.

Accordingly, $D$ is reducible in $F_0'$, the image of $F_0$ on $\Sigma'$, 
plus an irreducible cubic curve $Z'$ (that projects to $Z$). 
Let $E'$ be the image of $E$ on $\Sigma'$. 
Then the plane joining $y$ with $E'$ is tangent to $\Sigma$ all along $F'_0$. 
Moreover $Z'$  intersects the lines of the ruling of $\Sigma'$ in one point, 
and is therefore rational and intersects the line $F_0'$ transversally at one point $z$. 
Notice that $Z'$ cannot be a plane section of $\Sigma'$. 
In fact, if this would be the case, then $Q$ would be reducible in two planes: 
the plane containing $Z'$ plus the plane joining $y$ with $E'$ 
(because this is tangent to $\Sigma$ all along $F'_0$). 
But then $Q$ would contain $y$ and therefore $y$ would be on $\Sigma$, 
which is impossible. So $Z'$ is a rational normal cubic.

As in the case in which $D$ is irreducible, 
there are no proper secant lines to $D$ passing through the centre of projection $y$. 
This implies that the projection $Z$ of $Z'$ is cuspidal. 
Moreover $a$ and $Z$ must intersect in only one point $\zeta$, 
i.e., the projection of the point $z$. 
To finish the proof we have to show that it cannot be the case 
that $\zeta$ is the cusp for $Z$ and that $a$ is the cuspidal tangent. 
Indeed, if this is the case, then in a chart centered at $\zeta$, 
$B$ would have equation of the form $y^3=yx^3$. 
On the other hand, it is well known that $D$ has an equation of the form $B^2=AC$ 
where $A,B,C$ are suitable functions (see \cite [p. 1136]{M}) 
and $y^3=yx^3$ is not of this type.

We conclude that $\zeta$ is a smooth point for $Z$
and $a$ meets $Z$ only there, and hence it is the flex tangent to $Z$ at $\zeta$.
\end{proof}

When the scroll $\Sigma$ is smooth,
the general projection to a plane has an irreducible quartic (with three cusps) as its branch curve.  
To finish this subsection, we prove that the other case
(of a reducible branch curve)
can in fact occur:

\begin{proposition}\label{prop:true} If $X=\Sigma\subset \bbP^4$ is smooth, there are suitable triple plane projections $\pi: X\longrightarrow \bbP^2$ such that the branch curve consists of a cuspidal cubic curve plus the tangent to its only flex.
\end{proposition}

\begin{proof} 
Taking into account Proposition \ref {prop:pomp}, 
it suffices to prove that there are triple plane projections 
such that the branch curve contains a line. 
To see this, we argue as follows. 
The surface $\Sigma$ is a scroll, that has the line directrix $E$. 
Let $F_0$ be a line of the ruling and consider the plane $\langle E,F_0 \rangle$. 
Let $x$ be a general point of this plane 
and project $\Sigma$ down to $\bbP^3$ from $x$. 
The image of the projection is a cubic surface $\Sigma'$ 
with a line $r$ of double points 
whose pre--image on $\Sigma$ is the union of $E$ and $F_0$. 
This surface $\Sigma'$ is a so--called \emph{Cayley cubic scroll}. 
Notice that there is a $3$--dimensional linear space $\Pi$ 
which is tangent to $\Sigma$ along $F_0$ 
and contains the plane $\langle E,F_0 \rangle$, so it contains $x$. 
The projection of $\Pi$ from $x$ is a plane $\alpha$, 
whose intersection with $\Sigma'$ is the line $r$ counted with multiplicity $3$, 
i.e., $\alpha$ is tangent, along $r$, 
to one of the two branches of $\Sigma'$ intersecting along $r$. 
Let us take $q$ a general point of $\alpha$. 
It turns out that the projection from $q$ 
determines a triple plane map $\pi: X=\Sigma\longrightarrow \bbP^2$ 
whose branch curve consists of a cuspidal cubic curve 
plus its unique flexed tangent. 
This is clear from the proof of Proposition \ref {prop:pomp} 
(because the branch curve contains the line $a$, the projection of $r$) 
and it has been proved in all details in \cite [\S 36]{P1}. 
We will sketch the proof here for the reader's convenience. 

We can choose homogeneous coordinates $[x,y,z,t]$ in $\bbP^3$ 
so that $q$ is the point $[0,0,1,0]$. 
Then $\Sigma'$ has equation of the form
\begin{equation}\label{eq:trip}
z^3+{\bf a_1} z^2+{\bf a_2} z+ {\bf a_3}=0
\end{equation}
where $\bf a_i$ is a homogeneous polynomial of degree $i$ in $x,y,t$, 
for $1\leq i\leq 3$. 
By making the projective Tschirnhausen transformation 
sending $(x,y,z,t)$ to $(x,y,z-{\bf a_1}/3,t)$, 
we transform the equation \eqref {eq:trip} into an equation of the form
\begin{equation}\label{eq:tripo}
z^3+3{\bf b_2} z+ {\bf b_3}=0
\end{equation}
where $\bf b_i$ is a homogeneous polynomial of degree $i$ in $x,y,t$, 
for $2\leq i\leq 3$. 
To obtain the branch curve of the triple plane,
we consider the discriminant of this equation, 
i.e., $4{\bf b_2}^3+{\bf b_3}^2$. 
One has $4{\bf b_2}^3+{\bf b_3}^2={\bf r}^2 {\bf D}$, 
where ${\bf r}$ is a linear factor, 
whose zero set is the projection on the plane $z=0$ 
of the double line $r$ of $\Sigma'$, 
and ${\bf D}=0$ is the equation of the branch curve, that is in fact a quartic. 
We can actually suppose, up to a change of coordinates, 
that $r$ is the line $x=z=0$. 
Then setting $x=0$ into \eqref {eq:tripo}, 
we must have $z^3=0$, 
which means that $x$ has to divide both $\bf b_2$ and $\bf b_3$. 
Moreover if we put $z=0$ into \eqref {eq:tripo} 
we have that the resulting polynomial, namely ${\bf b_3}$,  
has to be divisible by $x^2$. 
So we have ${\bf b_2}=x{\bf c_1}$ and ${\bf b_3}=x^2{\bf d_1}$, 
with ${\bf c_1}, {\bf d_1}$ suitable linear forms. Hence 
$$
4{\bf b_2}^3+{\bf b_3}^2=4x^3{\bf c_1}^3+x^4{\bf d_1}^2=x^3(4{\bf c_1}^3+x{\bf d_1}^2)
$$
from which we see that ${\bf D}=x(4{\bf c_1}^3+x{\bf d_1}^2)$ 
is divisible by $x=0$; 
hence the branch curve consists of the line $r$ 
and of the cuspidal cubic with equation $4{\bf c_1}^3+x{\bf d_1}^2=0$, 
such that $x=0$ intersects it with multiplicity $3$ in the flex 
defined by $x={\bf c_1}=0$. 
\end{proof}

\begin{remark}\label{rem:all} 
We notice that all plane cuspidal cubics are projective equivalent 
and therefore also the unions of such a cubic with the tangent line in its only flex 
are projectively equivalent. 
Therefore any union of a cuspidal cubic curve 
 plus its flex tangent line 
is the branch curve of an otp that arises from a special projection of a smooth rational normal scroll $\Sigma\subset \bbP^4$. 
\end{remark}

In either case, the Tschirnhausen vector bundle of $\pi: X\longrightarrow \bbP^2$ 
is $\calO_{\bbP^2}(-1)\oplus \calO_{\bbP^2}(-1)$ 
(see \cite [Table 10.5]{M}).

\subsection{An alternate viewpoint}

Let $\bbF_1$ be the abstract relatively minimal rational ruled surface
with a negative section $E$, with $E^2 = -1$, and fiber class $F$.
It is the blowup of $\bbP^2$ at one point $\uz$
(the exceptional divisor is $E$),
and the class of a general line in the plane lifts to a section $S$ of the ruling, with $S^2=1$;
we note that $S \sim E+F$.

The embedding of $\bbF_1$ into $\bbP^4$ as the surface $\Sigma$
is given by the complete linear system $|E+2F| = |2S-E|$,
which corresponds to the linear system of conics through the point $\uz$.
This system has dimension $4$ and self-intersection $3$.

The map $\pi:\bbF_1 \to \bbP^2$ is then given by 
a base point free linear system $\calL \subset |E+2F|$ of dimension two, 
and corresponds to a net of conics in $\bbP^2$ all passing through the point $\uz$.
Formally, the map $\pi$ associates to any point $x\in\bbF_1$
the pencil of conics in $\calL$ that contain $x$,
and therefore maps $\bbF_1$ to $\check \calL$:
\[
\pi(x) = \{C \in \calL \;|\; x \in C\} \in \check \calL.
\]
If $P \in \check \calL$ is a pencil of conics, then
\[
\pi^{-1}(P) = \{x \in \bbF_1 \;|\; x\in C, \;\forall C \in P\} = \Base(P)
\]
namely the set of base points of $P$ (off $\uz$), which is the intersection of any two members of $P$  (off $\uz$), and therefore has length $3$: this is the triple cover of course.
(This set of base points is finite for every $P$, since otherwise
the set of base points would be a curve for this pencil,
and hence there would be base points for the overall system $\calL$.)

The elements of the complete linear system $|E+2F|$ break into four types:
\begin{itemize}
\item[(a)] Smooth members, corresponding to a smooth conic in $\bbP^2$ through $\uz$.
\item[(b)] A reducible curve of the form $S+F$, 
where $S$ is a smooth section with $S^2=1$ and $F$ is a fiber.  
This corresponds to a pair of distinct lines in $\bbP^2$, one of which passes through $\uz$.
\item[(c)] A reducible curve of the form $E+F_1+F_2$, 
consisting of the negative section $E$ plus two distinct fibers.  
This corresponds to a pair of distinct lines through $\uz$.
\item[(d)] A reducible curve of the form $E+2F$, 
consisting of the negative section $E$ plus the double of a fiber $F$.  
This corresponds to a double line through $\uz$.
\end{itemize}

\begin{lemma}
These four types are represented by four loci in the complete linear system.
These have the following dimensions and degrees.
\begin{itemize}
\item[(a)] This is the general member of $|E+2F|$, and these members form an open set in $|E+2F|$.
\item[(b)] The closure of the set of these members has dimension $3$; it is a hypersurface in $|E+2F|$ of degree $3$.  Its closure is the complement of the locus of members of type (a).
(This is the dual to the embedded surface $\Sigma$.)
\item[(c)] The closure of the set of these members has dimension $2$; it is a linear subspace of $|E+2F|$; it has degree $1$.
\item[(d)] The set of these members is a smooth conic in the plane of the (c) members.
It has dimension $1$ and degree $2$.
\end{itemize}
\end{lemma}

\begin{proof}
For (b), since the linear system $|S|$ has dimension two, the members of the form $S+F$ will have dimension three as claimed.  
The degree of this locus is the number of such members in a general pencil of conics, 
and that is three also: a pencil of conics through $\uz$
will have four base points $\uz, p, q, r$, and the  members
$\langle\uz, p\rangle +\langle q, r \rangle$,
$\langle \uz, q \rangle+\langle p, r\rangle$,
$\langle \uz, r\rangle +\langle p,q \rangle$
are the three singular members giving rise to elements of type (b).

For (c), the members are just the singular members of the system of conics through $\uz$, and this is a sub-linear system (of dimension two)
and therefore has degree one.
For (d), the double lines are a conic in that system.\end{proof}

Of course we have
\[
\overline{(d)} \subset \overline{(c)} \subset \overline{(b)} \subset \overline{(a)} = |E+2F|.
\]

Let us reconfirm that the branch curve for the triple cover $\pi$ has degree $4$.
Since $\pi^{-1}(P) = \Base(P)$ for a pencil $P$ on $\mathbb F_1$,
we see that $P$ is part of the branch curve if and only if there is an infinitely near base point for the pencil $P$ at some point $x\in \mathbb F_1$.
Suppose now that $L$ is a line in $\check \calL$;
then the number of pencils in $L$ with an infinitely near base point
measures the degree of the branch curve.
The line $L$ is determined by a given member (corresponding to a conic) $C \in \calL$:
if we let $L_C$ be the set of pencils inside $\calL$ containing $C$,
then that line $L$ is equal to $L_C$ for some conic $C\in\calL$.

Now inside the complete linear system $|E+2F|$ of conics through $\uz$,
the general pencil containing a smooth conic $C$
and having an infinitely near base point
can be uniquely generated by $C$ and a pair of lines,
one of which passes through $\uz$ and another point of $C$,
and the other of which is tangent to $C$  at some point of $C$.
So this set of pencils is determined by these two points of $C$,
and therefore corresponds to a surface inside $|E+2F|$,
parametrized by these two points of $C$.
This surface has degree four: if we parametrize the points of the conic by quadratics,
the tangent line and the other line are both quadratic expressions,
and their product is a quartic in the parameter.

Finally the number of elements of $L_C \subset \calL$ is just the number of intersection of the dimension two subspace $\calL$ with this degree $4$ surface in $|E+2F|$,
which is then four.

We note next that the base-point-free net $\calL \subset |E+2F|$ cannot lie inside the cubic hypersurface consisting of the closure of the (b) members, by Bertini.  Hence inside $\calL$ there will be a curve consisting of the singular members, and this curve is a cubic curve, which is  the dual of the branch curve, since a member of $\calL$ is the pullback of a line in the $\bbP^2$ and that member is singular if and only if that line is tangent to the branch curve.

In general, the branch curve is an irreducible quartic with three cusps; by the Pl\"ucker relations, its dual is a cubic curve (as noted above) with one node.
(That node corresponds to the single bitangent to the branch curve.)

The intersection of $\calL$ with the closure of the elements of type (c) is the intersection of two planes in the $\bbP^4$ which is $|E+2F|$.  This intersection cannot be a line,
since if it were, then $\calL$ would have two members of type (c), and these two members would then both contain the negative curve $E$.  Therefore any third member of $\calL$, which would then generate $\calL$ along with those first two, would meet $E$, and therefore $\calL$ would have a base point.

Hence $\calL$ meets $\overline{(c)}$ in just one point: there is a single member of $\calL$ of type (c).  That member is, in general, the pre-image of the bitangent line to the branch curve, which splits into the three irreducible components of the type (c) element.  In particular we see that, in general, the negative curve $E$ maps isomorphically to that bitangent line (as  do the two fibers of that type (c) element of $\calL$).

Again in general, $\calL$ will not contain any elements of type (d), for dimension reasons.

\begin{lemma}
$\calL$ contains an element of type (d)
if and only if the quartic branch curve contains a line as a component.
If so, it can only contain one element of type (d).
\end{lemma}

\begin{proof}
The proof is immediate: the element of type (d) has a double component,
and is the pre-image of a line $a$ in $\bbP^2$; 
hence the entire double component is a part of the ramification locus, 
and therefore its image is part of the branch locus.
Conversely, if the branch curve contains a line, the pre-image of that line is a member of $\calL$ with a double component, and that must be a member of type (d).

The final statement follows since if $\calL$ contains two elements of type (d), it would then contain two points in the $2$-plane of elements of type (c), which we have explained above does not happen.
\end{proof}

Let us analyze this case further. Suppose that $\calL$ contains a member of type (d).
We will use affine coordinates $(x,y)$ in $\bbP^2$ such that $\bbF_1$ is the blowup of the origin.
In that case the elements of $|E+2F|$ are represented by conics through $\uz=(0,0)$, and are therefore all in the linear system generated by the monomials $<x^2, xy,y^2,x,y>$.
The linear system $\calL$ has vector space dimension three,
and we may choose the coordinates so that the member of type (d)
is given by the polynomial $x^2$.
In that case we can arrange, using row reduction, that the three generators of $\calL$ are:
\begin{align*}
f &= x^2 \\
g &= xy + px + qy \\
h &= y^2 + rx + sy
\end{align*}
for suitable constants $p,q,r,s$. 
Note that $q\neq 0$ otherwise $\calL$ has base points off $	\uz$.
The triple cover map $\pi$ sends a point $(x,y)$ 
to the point with homogeneous coordinates
$[f(x,y):g(x,y):h(x,y)] \in \bbP^2$.
The points of the line $x=0$ map to points of the form $[0:qy:y^2+sy] = [0:q:y+s]$
which is the line component of the branch curve.

If $[U:V:W]$ are homogeneous coordinates on the target $\bbP^2$,
then we can switch to the affine coordinates $v=V/U,w=W/U$, and write the map $\pi$ as
\[
\pi(x,y) = (v, w) = (y/x + p/x + qy/x^2, (y/x)^2+r/x+sy/x).
\]
If we change generators for the field of functions $\bbC(x,y)$ to $\alpha=1/x,\beta= y/x$,
then this becomes
\[
v = \beta + p \alpha + q \alpha\beta, \;\;\;
w = \beta^2 + r \alpha + s \beta
\]
which defines then the field extension $\bbC(v,w) \subset \bbC(\alpha,\beta)$.
This is a cubic extension of course, 
and we can eliminate $\alpha$ from the two above equations to obtain
\[
r(v-\beta) = (p+q\beta)(w-\beta^2-s\beta)
\]
or
\begin{equation}\label{betacubed}
q\beta^3 + (qs + p)\beta^2 + (ps - qw - r)\beta + (rv - pw) = 0
\end{equation}
which is the defining cubic polynomial determining the field extension.

If one divides by $q$ and completes the cube here, by setting
\[
\gamma = \beta + \frac{s+p/q}{3}
\]
we will obtain a cubic equation of the form
\[
\gamma^3 + A\gamma + B=0
\]
where $A=A(w)$ is linear, and $B=B(v,w)$ is also linear.
Hence we can linearly change coordinates to $A$ and $B$,
and see that the discriminant $D = 4A^3+27B^2$ 
defines an irreducible cuspidal cubic.
The image of the line determined by $x=0$, 
which is the other component of the branch locus for the triple cover map $\pi$, 
is the line at infinity in these coordinates (since $\alpha = 1/x$); 
that line at infinity is the unique flex tangent to the cuspidal cubic.

Hence (see also the proof of Proposition \ref {prop:true}):

\begin{lemma} The net $\calL$ contains a member of type (d) if and only if 
the branch curve of the triple cover is of the form $a+Z$,
where $Z$ is a cuspidal cubic and $a$ is the unique flexed line to $C$.
\end{lemma}

\subsection {The structure constants viewpoint}\label{ssec:rot}
Finally let us analyze this case from the viewpoint 
of the structure constants of the algebra,
in the spirit of \cite{M}. 
Recall that the triple cover is determined by four linear polynomials $a,b,c,d$ 
(in affine coordinates on the target $\bbP^2$), 
using equations of the form
\begin{align*}
z_0^2 &= 2A + az_0 + bw_0 \\
z_0w_0 &= -B -dz_0-aw_0 \\
w_0^2 &= 2C + cz_0 + dw_0.
\end{align*}
where $A = a^2-bd$, $B=ad-bc$, and $C=d^2-ac$,
as in \eqref{triplecoverstructureconstants} below.

\begin{lemma}
The triple cover has branch locus equal to 
a cuspidal cubic plus its unique flex tangent line
if and only if
there are linear polynomials $a,b,c,d$ as structure constants for the triple cover
such that $a=b$.
\end{lemma}

\begin{proof}
Suppose first that $a=b$ as two linear polynomials.
We note that the discriminant of the triple cover (see \cite{M})
is $D = B^2-4AC$; hence if $a=b$, we have
$A=a(a-d)$, $B=a(d-c)$, and $C=d^2-ac$,
so that $a$ factors out of $D$, with the residual cubic being
$a(d-c)^2 - 4(a-d)(d^2-ac) = 4a^2c + ac^2 - 6acd - 3ad^2 + 4d^3$.
This cubic meets the line $L$ defined by $a=0$ only at $d=0$,
so the line $L$ is a flex tangent to this residual cubic.

Moreover the triple cover is totally ramified over the points where $A=B=C=0$,
and solving this we see that either $a=d=0$ 
(this is the flex point of intersection of $L$ and the cubic) 
or $a\neq 0$ so that $a=d=c$.
This latter point is the intersection of the two lines $a=d$ and $d=c$
and gives the cusp of the cubic.

Conversely, let us take the triple cover defined by the cubic extension as in \eqref{betacubed}, divide by $q$, and replace $p/q$ and $r/q$ by $p$ and $r$, to obtain
\[
\beta^3 + (s + p)\beta^2 + (ps - w - r)\beta + (rv - pw) = 0
\]
where $p,r,s$ are constants and $v,w$ are the affine coordinates in the target.

We may now apply Proposition \ref{TCstructure} in the Appendix
to find the triple cover structure constants;
using the notation introduced there, in the above situation we have
\[
e_0 = rv-pw; \;\;\; e_1 = ps-r-w; \;\;\; e_2 = s+p
\]
so that
\begin{align*}
a&=2(s+p)/3; \;\;\; b = 1; \;\;\; \\
c &= -rv-sw+(p+s)(p+s-r); \;\;\; d = (-1/3)w + (s^2+3ps+p^2-r)/3.
\end{align*}

What we see here is that the two structure constants $a$ and $b$
are actually constants, independent of the affine variables $v$ and $w$.
Therefore we can scale $\beta$ to make them equal as constants.
Then $a=b$ as desired; the zero locus is the line at infinity in these affine coordinates,
and gives the flex tangent line to the cuspidal cubic.
\end{proof}

\section{Branch curve of degree 6} \label{sec:6}
\subsection{The general analysis} 
In this section we treat the case where an otp $\pi: X\longrightarrow \bbP^2$ 
has branch curve of degree $6$. 
The general curve $C\in \calL$ is then smooth of genus $g=1$ with $C^2=3$.

\begin{proposition}\label{prop:3} 
Consider an otp $\pi: X\longrightarrow \bbP^2$ 
such that its branch curve has degree $6$. 
Then $X$ has Kodaira dimension $-\infty$ 
and either $X$ has irregularity $q=1$, or $X$ is rational. 
In the former case, either the triple plane is as in the basic exceptional example, 
or the branch curve consists of six lines in a pencil. 
In the latter case, $X$ is a normal cubic surface  in $\bbP^3$ 
and the triple plane map is the projection of $X$ to a plane from a point not on $X$. 
\end{proposition}

\begin{proof} 
One has $K_{\tilde X}\cdot C=-3$, 
which implies that $X$ has Kodaira dimension $-\infty$. 
If $X$ is irregular, 
then $q=1$ because $X$ contains a linear system of elliptic curves. 
If $q=1$, we apply Corollary \ref {cor:pomp}. 
Then either the triple plane is an exceptional example 
(and then it is the basic exceptional example 
because its branch curve has degree six) 
or it is Cremona equivalent to a triple plane $\pi': X'\longrightarrow \bbP^2$ 
as in diagram \eqref {eq:loh} 
with the branch curve consisting of the union of $2q+4=6$ lines in a pencil. 
Then the assertion follows from the final assertion of Corollary \ref {cor:pomp}.

If $X$ is regular, then it is rational 
and the complete linear system $|C|$ has dimension $3$ and the assertion follows. 
\end{proof}

Consider now the case in which the otp $\pi: X\longrightarrow \bbP^2$ 
is given by an external projection to a plane of a normal cubic surface. 
In this case the Tschirnhausen vector bundle is 
$\calO_{\bbP^2}(-1)\oplus \calO_{\bbP^2}(-2)$ 
(see \cite [Table 10.5]{M}). 
Let us determine the geometry of the branch curve 
in the case when $X$ is smooth.

\begin{proposition}\label{prop:sm} 
Let $\pi: X\longrightarrow \bbP^2$ be an otp 
given by an external projection to a plane of a smooth cubic surface. 
Then either the branch curve $B$ is an irreducible sextic  that has in general six ordinary cusps 
or it is reducible, consisting of two smooth  cubics 
that intersect in three collinear points with intersection multiplicity $3$ at each of them. 
\end{proposition}

\begin{proof} 
Let $p$ be the external point to $X$ from which we make the projection 
and let $D$ be the intersection of $X$ with the polar quadric of $p$ with respect to $X$. 
The curve $D$ is the ramification locus of the triple cover map $\pi$, and 
the branch curve $B$ is the projection of $D$ from $p$. 
Since $D$ is the intersection of a cubic and a quadric surface,
it has degree six in $\bbP^3$, as does the projection $B$ in $\bbP^2$.
The arithmetic genus of $D$ is $4$. 

We claim that $B$ cannot contain a line or an irreducible conic. 

Indeed, if $B$ contains a line $a$, 
then $D$ contains a line $b$ that projects to $a$. 
Then the pull--back via $\pi$ of $a$ consists of $2b$ plus another line.
This means that there is a plane tangent to $X$ along a line,
which cannot be the case if $X$ is smooth. 

Suppose that $B$ contains an irreducible conic $\Gamma$, 
that is the image of an irreducible conic $\Gamma'$ contained in $D$. 
Then the pull--back of $\Gamma$ via $\pi$ consists of $2\Gamma'$ 
plus another irreducible conic $\Delta$. 
Then $2\Gamma'+\Delta\in |\calO_X(2)|$. 
The plane containing $\Gamma'$ intersects $X$ along $\Gamma'+s$, 
where $s\subset X$ is a line. 
Then $2\Gamma'+2s\in |\calO_X(2)|$ and therefore $\Delta\sim 2s$. 
This is not possible, because $s^2=-1$ on $X$ whereas $\Delta^2=0$. 

Since $B$ cannot consist of a double cubic 
because we are assuming that the cover is not cyclic (see Remark \ref {rem:lpi}),
then either $B$ is irreducible 
or it consists of the union of two irreducible cubics. 
In any event $B$ is reduced.

Then, by Corollary 5.8 of \cite{M},
we have that the singularities of $B$ are only points with multiplicity two on $B$;
they are therefore all of type $a_k$ 
(locally analytically defined by an equation of the form $x^2=y^{k+1}$) 
for some $k \geq 1$.

Suppose first that the branch curve $B$ is irreducible. 
Then $D$ is irreducible of arithmetic genus $4$. 
If the centre of projection is general enough, 
then $D$ is non--singular by Bertini's Theorem. 
With the usual argument by now 
(see the proof of Proposition \ref {prop:pomp}), 
we see that no proper secant line to $D$ passes through $p$;
hence the projection $B$ of $D$ has only unibranch, i.e., cuspidal singularities.
Therefore they are all of type $a_{2k}$; these are cusps of order $k$.
We claim that $B$ cannot have cusps of order higher than one:
all singularities are ordinary cusps of type $a_2$.
Indeed, $B$ cannot have a cusp of order $3$ or more 
(of type $a_{2k}$ with $k \geq 3$) 
because the tangent line at a cusp of order at least $3$ 
would have intersection multiplicity at least $7$ at the cusp with $B$, which is a sextic. 
If it has a cusp of order $2$, namely of type $a_4$, 
then the tangent to $B$ at the cusp has intersection multiplicity $5$ with $B$ at the cusp, 
and so it intersects $B$ only in one more point. 
The pull back of this line 
(that is an irreducible curve because at the cusp we have total ramification) 
has therefore a $g^1_3$ that is ramified only at $2$ points 
(one simple and one double ramification),  a contradiction. 
So $B$, that has geometric genus $4$ since $D$ is its normalization, 
has exactly $6$ ordinary cusps of type $a_2$ as singularities.

Suppose next that $B$ is reducible, 
the union of two irreducible cubics $C_1,C_2$, so that $D$ is reducible. 
Since no proper secant line to $D$ can pass through $p$, 
then every irreducible component of $D$ 
maps birationally to an irreducible component of $B$ 
and every irreducible component of $B$ 
is the projection of a unique irreducible component of $D$. 
Hence $D=C_1'+C_2'$, with $C_1', C_2'$ two irreducible cubic curves 
that are respectively mapped to $C_1$ and $C_2$. 

We claim that $C_1', C_2'$ are both plane curves. 
Suppose this is not the case and assume that $C'_1$ is a rational normal cubic. 
Then the pull back of $C_1$ on $X$ consists of $2C'_1+Z$,
where $Z$ is another irreducible cubic curve that projects to $C_1$. 
If $Z$ is a plane cubic, then $2C'_1\in |\calO_X(2)|$,  
which is not possible because $(C'_1)^2=1$ 
whereas the self intersection of members of $|\calO_X(2)|$ is $12$.  
Therefore $Z$ is also a rational normal cubic. 
Now recall that $X$ is the image of the blow--up $\widetilde {\bbP^2}$ of $\bbP^2$ 
at six distinct points $p_1,\ldots, p_6$ (no three collinear, not on a conic) 
via the pull back on $\widetilde {\bbP^2}$ of the linear system $|3;1^6|$. 
We can assume that $C'_1$ is the image of the pull back on $\widetilde {\bbP^2}$ 
of a line $L$ of $\bbP^2$ not passing through any of the points $p_1,\ldots, p_6$. 
Then $Z$ would be in the linear system $|7;3^6|$. 
But the system $|7;3^6|$ is empty: 
indeed, $|7;3^6|$ is Cremona equivalent to $|5;3^3, 1^3|$,
and for this system we see that the three lines joining the points with multiplicity $3$ split off and the residual system $|2;1^6|$ is empty. 
 
We conclude that $C_1', C_2'$ are plane curves 
respectively isomorphic to $C_1$ and $C_2$
via the projection from $p$. 
Then the polar quadric of $p$ with respect to $X$ 
is the union of the two planes containing $C_1'$ and $C_2'$.
We remark that $C_1', C_2'$ are both smooth 
and therefore also $C_1$ and $C_2$ are  smooth. 
Indeed, if, say, $C'_1$ would be singular at a point $q$ 
then the plane containing $C'_1$ would be tangent to $X$ at $q$. 
But then this plane should contain $p$, which is not possible.

To finish the proof we make a specific computation.  As in the proof of Proposition \ref {prop:true}, 
we may assume that $p$ is the point $[x:y:z:t] = [0:0:1:0]$, 
that $X$ has equation 
\[
f=z^3-3{\bf b_2} z+ {\bf b_3}=0,
\]
and that the $\bbP^2$ on which we project has equation $z=0$. 
The polar quadric of $p$ with respect to $X$ has equation $z^2={\bf b_2}$. 
Since this quadric has to be reducible in two distinct planes, 
${\bf b_2}$ has to be the square of a linear form. 
Up to a change of coordinates, we can assume that ${\bf b_2}=x^2$. 
Hence the two planes forming the polar quadric of $p$ 
have equations $z+x=0$ and $z-x=0$, 
which intersect along the line $x=z=0$. 
The branch curve has equation $4{\bf b_2}^3-{\bf b_3}^2=0$, 
hence the two cubics $C_1$ and $C_2$ on the plane $z=0$ have equations 
\[
{\bf b_3}-2 x^3=0, \quad {\bf b_3}+2x^3=0.
\]
We claim that the line $x=z=0$ intersects $X$ in three distinct points. 
In fact the line $x=z=0$ intersects $X$ in the points with $x=z=0$ 
and ${\bf b_3}(y,0,t)=0$. 
Suppose that one of these points $q=(0,\eta,0,\tau)$ appears with multiplicity at least two. Then we have
$$
\begin{aligned}
&\frac  {\partial f} {\partial x}(q)=\frac {\partial {\bf b_3}} {\partial x}(\eta,0,\tau)=0\\
&\frac  {\partial f} {\partial y}(q)=\frac {\partial {\bf b_3}} {\partial y}(\eta,0,\tau)=0\\
&\frac  {\partial f} {\partial z}(q)=0\\
&\frac  {\partial f} {\partial t}(q)=\frac {\partial {\bf b_3}} {\partial t}(\eta,0,\tau)=0
\end{aligned}
$$
and so $q$ would be a singular point for $X$, a contradiction. 
Hence the curves $C_1$ and $C_2$ 
intersect in three distinct points along the line $x=z=0$.

We finally claim that $C_1$ and $C_2$ intersect at each of these three points 
with intersection multiplicity $3$. 
Indeed, since no proper secant line to the ramification curve $D$ passes through $p$,
then $C_1$ and $C_2$ have no intersection point off the three aforementioned points. 
To prove the claim, we first assume ${\bf b_3}$ to be a general polynomial of degree $3$ 
and call $m_i$ the intersection multiplicity of $C_1$ and $C_2$ at the three points, 
with $1\leq i\leq 3$. 
Now specialize ${\bf b_3}$ so that ${\bf b_3}(y,0,t)=y^3-t^3$. Under this specialization, $C_1$ and $C_2$ specialize to the two cubics with equations $y^3\pm 2x^3-t^3=0$, 
that intersect at the three points $(1,0,1)$, $(\pm \xi,0,1)$, 
with $\xi$ a primitive cubic root of $1$. 
A direct computation shows that in this case $C_1$ and $C_2$  
have intersection multiplicity $3$ at each of the three points. 
As a consequence, if ${\bf b_3}$ is a general polynomial of degree $3$
we have $m_1=m_2=m_3=3$ 
and therefore the same happens for every specialization of ${\bf b_3}$. 
\end{proof}

\begin{remark} \label{rem:op} 
(i) We notice that the proof of Proposition \ref {prop:sm} 
shows that there are otp's $\pi: X\longrightarrow \bbP^2$ 
given by an external projection to a plane of a smooth cubic surface 
with branch curve reducible in two smooth cubics.  
Actually we expect that, for a general smooth cubic surface $X$ in $\bbP^3$, 
there are exactly $10$ points $p\not\in X$ 
such that the projection of $X$ from $p$ gives rise to a triple plane 
with branch curve $B$ reducible in two smooth cubics. 
Indeed, given $X$, look at the $3$--dimensional linear system $\calF$ 
of polar quadrics of $X$. 
Inside $\calF$ there is the discriminant surface $\calX$, 
whose points correspond to singular quadrics. 
The surface $\calX$ has degree $4$ 
and it is defined by the vanishing of a $4\times 4$ symmetric determinant of linear forms. 
As a consequence $\calX$ has  exactly $10$ double points 
that correspond to rank $2$ quadrics. 
Accordingly there are $10$ points $p$ such that 
the polar quadric of $p$ with respect to $X$ is reducible in two distinct planes. 
The projection of $X$ from such points $p$ 
give rise to otp's with branch curve reducible in two smooth cubics. \smallskip

(ii) In the statement of Proposition \ref {prop:sm}, 
we claim that if the branch curve $B$ is irreducible, 
then \emph{in general} it has six ordinary cusps. 
We want to stress that the generality assumption is essential for the conclusion. 
Indeed (keeping the notation of the proof of Proposition \ref {prop:sm}), 
it may happen that in particular cases the ramification curve $D$ is irreducible 
but singular, it has perhaps a node. 
In that case in general the branch curve $B$ will accordingly have a tacnode
in correspondence with the node of $D$, not a cusp. 
Since we may expect that in the $3$--dimensional linear system of curves $D$ 
cut out on $X$ by its first polars there are curves $D$ with three nodes, 
we may also expect that there are special projections of $X$ 
for which the branch curve is irreducible with three tacnodes instead of six cusps. 
However we will not dwell here 
on the problem of which configurations of singularities the branch curve may have 
for special projections. 

As for the presence of points $x\not\in X$ 
for which the ramification curve $D$ has a node, 
remember that we know that there are points $p\not\in X$ 
such that the polar of $X$ with respect to $p$ splits in two planes, 
whose intersection is a line 
intersecting $X$ in three distinct points $x_1,x_2, x_3$, 
such that the ramification curve $D$ of the projection of $X$ from $p$ 
consists of two plane cubics intersecting at $x_1,x_2, x_3$. 
Now let us consider the line $\langle p, x_1\rangle$. 
The polars of the points on this line cut out on $X$ a pencil of curves 
and all these curves are singular at $x_1$ 
and the general one has a node at $x_1$. 
Indeed, both the polars of $p$ and of $x_1$ have a singularity in $x_1$ 
and the polar of $p$ has a node at $x_1$. 
So the polar of a general point $x\in \langle p, x_1\rangle$ 
cuts on $X$ an irreducible curve $D$ with a node at $x_1$. 

(iii) Returning to the cubic equation $z^3+A(x,y)z+B(x,y)=0$ for the triple plane,
where $A$ is a quadratic and $B$ is a cubic,
then when $A=0$ and $B=0$ meet transversally the branch curve
(defined by the discriminant $D = 4A^3+27B^2$)
clearly has an ordinary cusp.
Suppose on the contrary that $A=0$ and $B=0$ are tangent (to order two).
Choosing analytic coordinates at the point of intersection,
we may assume that $A=x$ and $B = x+y^2$;
in that case $D = 4x^3 + 27(x+y^2)^2=0$ 
has a double tangent ($x=0$)
and after one blowup it has a cusp (of type $a_2$).
Hence $D=0$ has a singularity of type $a_4$.

The type of singularity of the branch locus is determined by the nature of the intersections between the conic and the cubic; to obtain three $a_4$ singularities we use a conic that is tritangent to the cubic, for example.  Higher-order cusps are possible with higher-order tangencies between the conic and the cubic; we will not seek an exhaustive analysis here.
\end{remark}

The case in which $X$ is a singular normal cubic surface is more complicated 
and we will not delve here into the minute classification 
of all possible cases that show up for the branch curve of a triple plane 
given by an external projection of such a surface. 
We remark that, if $X$ has $\delta$ ordinary double points 
(with $1\leq \delta\leq 4$) 
in general such a branch curve will be a sextic with $6$ cusps and $\delta$ nodes.

\subsection{The structure constants viewpoint}\label{ssec:iol} 
From the perspective of the structure constants for the triple cover
as developed in \cite{M},
we see that in the case of an external projection of a cubic surface in $\bbP^3$ to a plane,
we have, as we said, that the Tschirnhausen vector bundle is isomorphic to 
$\calO_{\bbP^2}(-1)\oplus\calO_{\bbP^2}(-2)$.
In the notation of \cite{M},
if we take $z$ to be the local coordinate for the $\calO_{\bbP^2}(-1)$ summand
and $w$ for the $\calO_{\bbP^2}(-2)$ summand,
then the structure constants $a,b,c,d$ 
(see \eqref{triplecoverstructureconstants} below)
for the triple cover have degrees
$1$, $0$, $3$, $2$ respectively.
The constant $b$ cannot be zero, and we can scale it to be equal to $1$;
this leads to the algebra being defined by equations of the form
\begin{align*}
z^2 &= 2A + az + w \\
zw &= -B -dz-aw \\
w^2 &= 2C + cz + dw.
\end{align*}
where $A = a^2-d$, $B=ad-c$, and $C=d^2-ac$;
these have degrees $2$, $3$, and $4$ respectively.
Using the first equation to solve for $w (=z^2-2A-az)$, 
and inserting that into the second equation,
gives the relationship
\[
z(z^2-2A-az) = -B -dz -a(z^2-2A-az)
\]
which simplifies to
\[
z^3 + (-3A)z + (B-2Aa) = 0.
\]
This is the equation of the cubic surface.
Hence in the notation above, we have
\[
{\bf b_2} = A = a^2-d \;\;\;\text{ and }\;\;\; {\bf b_3} = B-2Aa = ad-c-2a(a^2-d) = 3ad-2a^3-c.
\]
The branch curve is given by the polynomial 
\begin{align*}
D&=B^2-4AC = (ad-c)^2-4(a^2-d)(d^2-ac)=\\
&= a^2d^2-2acd+c^2-4a^2d^2+4d^3+4a^3c-4acd=\\
&= -3a^2d^2 -6acd +c^2 +4d^3 +4a^3c
\end{align*}
which is the sextic.

In general, this is a sextic with $6$ cusps.
We have seen above that the special case of two cubics 
meeting at three points with intersection multiplicity $3$ 
occurs exactly when ${\bf b_2}=a^2-d$ is a perfect square,
which we may assume is not zero since the cover is assumed to be not cyclic.
If we choose coordinates in the plane so that $a^2-d = x^2$,
then we have that $d = a^2-x^2$, so that
$A = {\bf b_2} = x^2$, $B= a^3-ax^2-c$, and $C = a^4-2x^2a^2+x^4 -ac$;
we also have ${\bf b_3} = a^3-3ax^2-c$.
In that case the branch curve is defined by
\[
D=B^2-4AC=(a^3-ax^2-c)^2-4x^2(a^4-2x^2a^2+x^4 -ac)
\]
which factors as
\[
(a^3 - 3ax^2 + 2x^3 - c)(a^3 - 3ax^2 - 2x^3 - c);
\]
these factors define the two cubic curves forming the branch curve.
Their intersection occurs at the three points where $x=0$ and $a^3=c$.
We note (consistent with the prior notation) 
that these two cubics are given by the equations
${\bf b_3} \pm 2x^3 = 0$.


We have already seen above that the polynomial ${\bf b_3}(0,y,t)$ has three distinct roots, i.e., the cubic ${\bf b_3} = 0$ meets the line $x=0$ in three distinct points.
Therefore the two cubics ${\bf b_3} \pm 2x^3 = 0$ meet in three distinct points;
since the pencil they generate has the triple line $x^3=0$ as a member,
the base points of that pencil are the three points, each with multiplicity three.\medskip

\section{Branch curve of degree 8}\label{sec:8}

\subsection{Generalities}
Consider now a triple plane $\pi: X\longrightarrow \bbP^2$ 
such that its branch curve has degree $8$.
Such a triple plane is certainly not cyclic, so it is an otp. 
Then the general curve $C\in \calL$ is smooth of genus $g=2$ with $C^2=3$. 

\begin{proposition}\label{prop:4}
Consider a triple plane $\pi: X\longrightarrow \bbP^2$ 
such that its branch curve has degree $8$. 
Then $X$ has Kodaira dimension $-\infty$ 
and $X$ has irregularity $q\leq 2$. 
If $q=2$ the branch curve of the triple plane consists of eight  lines in a pencil. 
\end{proposition}

\begin{proof} 
One has $K_X\cdot C=-1$, 
which implies that $X$ has Kodaira dimension $-\infty$
(and hence $p_g=0$). 
If $X$ is irregular, one has $q\leq 2$ 
because $X$ contains a linear system of genus two curves. 

Suppose  that $q=2$ and consider, as usual, the minimal desingularization $\phi: \tilde X\longrightarrow X$  of $X$ (see Section \ref {sec:prob}). Since $X$ has $q=2$ and $\kappa=-\infty$, we have a surjective morphism  $a:  \tilde X\longrightarrow G$ where $G$ is a smooth curve of genus 2 and the general fibre $F$ of $a$ is a $\bbP^1$. If $C$ is a general curve in $\tilde {\calL}$, that has genus 2, then $a_{|C}: C\longrightarrow G$ is an isomorphism, hence $C\cdot F=1$. So in the map $\pi\circ \phi: \tilde X\longrightarrow \bbP^2$, the fibres of the map $a$ are mapped to lines of a family $\calF$. 

Now we invoke Corollary \ref {cor:pomp}. Then either $\pi: X\longrightarrow \bbP^2$ is Cremona equivalent to an exceptional example, or the branch curve of
$\pi: X\longrightarrow \bbP^2$ consists of $8$ lines in a pencil. So, to finish the proof, we have to exclude the former case. In that case in fact we have a diagram 
\begin{equation}\label{eq:ex}
\xymatrix{
X \ar[d]_{\pi}\ar@{-->}[r] &
\Gamma[2]\ar[d]\\
\bbP^2 \ar@{-->}[r]^{\phi}& \bbP^2   }
\end{equation}
(where $\Gamma$ is a smooth elliptic curve and and $\Gamma[2]\longrightarrow \bbP^2$ is the basic exceptional example) and the rational map $\phi: \bbP^2 \dasharrow \bbP^2$ maps the lines of the family $\calF$ to the tangent lines to the branch curve of the basic exceptional example. This implies that $\phi$ is a projectivity and therefore $X\cong \Gamma[2]$ which is impossible because $q=2$. 
\end{proof}

\begin{example} \label{ex:lpo} 
To give a specific example of an otp 
with branch curve of degree $8$ and $q=2$ that is not a cone, 
look at the following situation. 
There are quintic surfaces $S$ in $\bbP^3$ 
with two skew lines $r,s$ of singular points, $r$ of double points, $s$ of triple points 
(see for instance \cite [p. 477]{Co}). 

The construction of such a surface is straightforward. 
In fact, take a smooth irreducible curve $C$ of genus $2$, 
and consider the canonical $g^1_2$ of $C$ 
and a base point free linear series $g^1_3$ on $C$. 
Then take two skew lines $r,s$ of $\bbP^3$ 
and map $C$ to $r$ via the $g^1_2$ and $C$ to $s$ via the $g^1_3$. 
Then for each point of $C$ take the corresponding points on $r$ and $s$ 
and join them with a line. 
The union of these lines is the desired quintic surface, 
and it is a scroll, with sectional genus $2$. 

The projection of $S$ from a general point $x$ of $r$ to $\bbP^2$ 
induces a triple plane $\pi: X\longrightarrow \bbP^2$ 
where $X$ is a smooth surface birational to $S$. 
The reader can readily check that 
the projection $\phi: S\dasharrow \bbP^2$ of $S$ from $x$ 
contracts $r$ to a point $p$ 
and the ruling of $S$ is mapped three to one to the pencil of lines with centre $p$, 
so that the branch curve consists of $8$ lines through $p$. 
\end{example}

\subsection{Exclusion of the $q=1$ case}
 
\begin{proposition}\label{prop:deg=8notq=1} 
Consider a triple plane $\pi: X\longrightarrow \bbP^2$ 
such that its branch curve has degree $8$. 
Then $X$ cannot have irregularity $q=1$.
\end{proposition}

\begin{proof} 
Let us work on the minimal desingularization $\phi: \tilde X\longrightarrow X$ of $X$, 
with the usual notation. We need a few preliminaries.

\begin{claim} \label{cl:uno} 
$K_{\tilde X}+C$ is nef.
\end{claim} 

\begin{proof} [Proof of the Claim \ref {cl:uno}] 
Since $C$ is big and nef,
one has  $h^i(\tilde X,\calO_{\tilde X}(K_{\tilde X}+C))=0$, for $i=1,2$, 
by the Kawamata--Viehweg theorem. 
Then by the Riemann--Roch theorem 
$h^0(\tilde X,\calO_{\tilde X}(K_{\tilde X}+C))=1$, 
so that $K_{\tilde X}+C$ is effective.
If there is an irreducible curve $\theta$ such that $\theta\cdot (K_{\tilde X}+C)<0$, 
one has $\theta^2<0$ and $K_{\tilde X}\cdot \theta <  0$ since $C$ is nef;
therefore  $\theta$ is a $(-1)$--curve, and hence $K_{\tilde X} \cdot \theta = -1$.  
This forces $C\cdot \theta=0$, which is impossible in our situation. 
\end{proof}

\begin{claim} \label{cl:due} 
Either $\tilde X$ is minimal, 
or $\tilde X$ is the blow--up of a minimal ruled elliptic surface $X'$ in one point. 
In this case there is a unique fibre of the Albanese pencil of $\tilde X$ 
that is reducible and it consists of two irreducible $(-1)$--curves. 
\end{claim} 

\begin{proof} [Proof of the Claim \ref {cl:due}]
By Claim \ref  {cl:uno}, we have
$$
0\leq (K_{\tilde X}+C)^2=K_{\tilde X}^2+4(g-1)-C^2=K_{\tilde X}^2+1.
$$
Then $K_{\tilde X}^2\geq -1$ and so either  $\tilde X$ is minimal and $K_{\tilde X}^2=0$ 
or $K_{\tilde X}^2= -1$ 
and $\tilde X$ is the blow--up of a minimal ruled elliptic surface in one point. 
\end{proof}

Let now $\eta\in {\rm Pic}^0(\tilde X)$ be any element, and let us consider the divisor $C_\eta=C-\eta$. Since $\chi (\calO_{\tilde X})=0$ and since $h^2(\tilde X, \calO_{\tilde X}(C_\eta))=0$, by Riemann--Roch theorem we have
$$
h^0(\tilde X, \calO_{\tilde X}(C_\eta))\geq \chi(\calO_{\tilde X}(C_\eta))= \frac {C_\eta(C_\eta-K_{\tilde X})}2=2,
$$
so that $C_\eta$ is effective for all $\eta\in {\rm Pic}^0(\tilde X)$. Moreover the curves
$C_\eta$  are nef and big as well as $C$, so they are 1--connected (see \cite[Lemma 2.6] {ML}). 

Consider the exact sequence 
$$
0\longrightarrow \eta \longrightarrow \calO_{\tilde X}(C)\longrightarrow \calO_{C_\eta}(C)\longrightarrow 0,
$$
with $\eta$ non--trivial. We claim that $h^i(\tilde X, \eta)=0$, for $0\leq i \leq 2$. Indeed, 
$h^0(\tilde X, \eta)=0$ is obvious and $h^2(\tilde X, \eta)=h^0(\tilde X, K_{\tilde X}-\eta)=0$, because $C\cdot  (K_{\tilde X}-\eta)=-1$ and $C$ is nef. Since $\chi(\eta)
=0$, the assertion follows. As 
$h^0(C_\eta, \calO_{C_\eta}(C))=3$,  we have a base point free $g_3^2$ on $C_\eta$ that is a 1--connected curve of genus $2$. This implies that $C_\eta$ cannot be irreducible. Let us pick up a point $x$ on an irreducible component of $C_\eta$ on which the $g_3^2$ has positive degree. Imposing to the divisors of the $g_3^2$ to contain that point, we get a residual $g^1_2$ on $C_\eta$ that has non--negative degree on any irreducible component of $C_\eta$. 

Let $\calN$ be the line bundle corresponding to the above $g^1_2$. 

\begin{claim} \label {cl:0} We can choose the point $x$ as above so that there is an irreducible component $Z$ of $C_\eta$ such that $\deg(\calN_{|Z})=0$. 
\end{claim}

\begin{proof} [Proof of Claim \ref {cl:0} ] If for any irreducible component $Z$ of $C_\eta$, one has $\deg(\calN_{|Z})>0$, since $C_\eta$ is reducible and $\deg(\calN)=2$, we have that $C_\eta$ splits in two irreducible components $Z_1+Z_2$ such that $\deg (\calN_{|Z_i})=1$, for $1\leq i \leq 2$.  
We cannot have $Z_1=Z_2$ since $C_\eta^2=3$. So $Z_1\neq Z_2$. Suppose that the point $x$  sits on $Z_1$. Then the original $g^2_3$ has degree 2 on $Z_1$ and degree 1 on $Z_2$, so that we could have chosen $x$ on $Z_2$, proving the Claim.\end{proof}

\begin{claim}\label{cl:3} The curve $C_\eta$ is not 2-connected.
\end{claim}

\begin{proof} [Proof of the Claim \ref {cl:3}] Suppose, by contradiction, that $C_\eta$ is  2-connected.

We can apply  Proposition (A.5) of \cite {CFM} to the line bundle $\calN$. 
If $C_\eta$ is 2-connected, then we have a decomposition $C_\eta=A+B$, with $A$, $B$ effective (and a priori depending on $\eta$), $A\cdot B= 2$ and $\calN_{|A}=\calO_A(B)$ (so that $\deg(\calN_{|A})=2$) and $\deg(\calN_{|B})=0$. Moreover, since $p_a(C_\eta)=2$, we have $p_a(A)+p_a(B)=1$ and $p_a(A), p_a(B)$ are  non--negative because $A$ and $B$ are both $1$--connected by \cite [Lemma (A.4)]{CFM} so either $p_a(A)=0, p_a(B)=1$ or $p_a(A)=1, p_a(B)=0$. 

Assume first $p_a(A)=0, p_a(B)=1$. Then the point $x$ can a priori lie either on $A$ or on $B$. Suppose first that it lies on $A$. Then  the original $g^2_3$ is concentrated on $A$ and trivial on $B$, hence $C\cdot B=0$. Now $|C-B|$ is a linear pencil, that, by the above argument, consists of rational curves, and this is a contradiction. So the point $x$ has to lie on $B$ and $C\cdot B=1$, $C\cdot A=2$. 
Since $3=(C_\eta)^2 = A^2 + B^2 +4$, we must have $A^2+B^2 = -1$.
Intersecting $C_\eta$ with $A$ we find that $2=A^2+ A\cdot B = A^2 +2$ so that $A^2=0$ and $B^2 = -1$. 
We conclude also that $A$ is a fiber of the Albanese pencil of $\tilde X$.

Since $C\cdot B=1$, there is an irreducible, rational component $\theta$ of $B$ such that $C\cdot \theta=1$ and $C\cdot (B-\theta)=0$. The curve $\theta$, being rational, sits in a fibre of the Albanese pencil of $\tilde X$. However, $\theta$ cannot be a fibre of the Albanese pencil, because $C\cdot \theta=1$ and $C$ has genus $2$. So, by Claim \ref {cl:due}, $\tilde X$ is the blow--up $f: \tilde X\longrightarrow X'$, of a minimal elliptic ruled surface $X'$ at a point $y$, obtained by contracting $\theta$. Let us denote by $C'$, $A'$, $B'$ the images of $C, A, B$ via $f$ on $X'$. 
Since $A\cdot \theta =0$, we have that $(A')^2=0$, and $(A')\cdot (B') = 2$.
We have 
$C'\equiv A'+B'$, 
$(C')^2=4$, and hence $(A')^2 + (B')^2 = 0$;
therefore $(A')^2=(B')^2 = 0$ and $A'$, $B'$ are still $1$--connected (this is a trivial computation, see \cite [\S 3] {F}). 

We have $X'=\mathbb P(\calE)$, where $\calE$ is a normalized rank 2 vector bundle on a smooth elliptic curve $G$, with invariant $e\geq -1$  and $e\geq 0$ if $\calE$ is decomposable (see \cite [Thm. 2.12, p. 376 and Thm. 2.15, p. 377]{har}).  Let us denote by $E$ a section with $E^2=-e$ and by $F$ a curve of the Albanese pencil of $X'$. Then $K_{X'}\equiv -2E-eF$. We have $C'\equiv aE+bF$, with $a,b$ suitable integers and $a\geq 2$. Imposing that
$$
4=(C')^2=a(2b-ae),
$$
we find $a=2, b=e+1$. Then $1-e= C'\cdot E\geq 0$, implies $e\leq 1$.

Next we write $B'=cE+dF$, with $c,d$ suitable integers with $c>0$. We have
$$
0 = (B')^2=c(2d-ce)
$$
so that $2d = ce$. 
Taking into account that, since $p_a(A)=0$, one has $A'\equiv F$ and we have
$$
2E+(e+1)F\equiv C'=A'+B'=cE+(d+1)F,
$$
so that $c=2$ and $d=e$.
Recall that $e\leq 1$. 

Assume first $e=-1$, so that $\calE$ is indecomposable 
and $X'$ is the symmetric square of the elliptic curve $\Gamma$, 
i.e., $X'\cong \Gamma(2)$, 
the curves $E\cong \Gamma$ are the \emph{coordinate curves} 
of the form $E_x=\{x+y\}_{y\in \Gamma}$ for all $x\in \Gamma$ 
and the fibres of the Albanese pencil are the curves $F$ 
described by all degree $2$ divisors $x+y$ that vary in a $g^1_2$ on $\Gamma$. 
In this case $C'\equiv 2E$. 
Now $|C'|$ has dimension $2$  
and it has no base points (see \cite [Thms. (1.17) and (1.18)] {CC}). 
This implies that the linear subsystem of $|C'|$ 
consisting of the curves in $|C'|$ containing the point $y$ 
has dimension $1$ and therefore $|C|$ cannot have dimension $2$, 
so that this case is excluded: $e \neq -1$. 

Assume next that $e=0$. 
Then $B'\equiv 2E$ which is not $1$--connected, because $E^2=0$; 
so this case is not possible. 
Similarly, in the case $e=1$ we must have $B'=2E+F$ which is not $1$--connected 
because $E\cdot (E+F)=0$, so also this case is also not possible.

This finishes the analysis when $p_a(A)=0$ and $p_a(B)=1$.

Assume next that $p_a(A)=1, p_a(B)=0$. 
Again the point $x$ can a priori lie either on $A$ or on $B$. 
Suppose first that it lies on $A$. 
Then  the original $g^2_3$ is concentrated on $A$ and trivial on $B$. 
Hence $C\cdot B=0$ so that $B^2<0$. 
Since $0=C\cdot B=A\cdot B+B^2=2+B^2$, 
we have $B^2=-2$ and $C\cdot A=3$ so that $A^2=1$.  
So $B$ should be a $1$--connected, rational curve with $B^2=-2$, 
and $B$ should be contained in a fibre of the Albanese pencil;
this is a contradiction to Claim \ref {cl:due}. 

We may assume then that $x$ lies on $B$ 
so that $C\cdot A=2$ and $C\cdot B=1$ (hence $B^2 = -1$).
Therefore  by Claim \ref{cl:due}, $B$ is an irreducible $(-1)$-curve 
and $\tilde X$ is not minimal; 
we contract it to $X'$ as above (and use the same notation as above), 
and $B$ is contracted to a point $y\in X'$.
The contractions $C'$ and $A'$ of $C$ and $A$ are now homologous;
we have $(C')^2 =4$.

Assume first $e=-1$. 
As we have seen above,  in this case we have $\dim(|C'|)=2$ 
and $|C'|$ has no base points. 
Then the linear subsystem of $|C'|$ consisting of the curves in $|C'|$ containing $y$ 
has dimension $1$ and therefore $|C|$ cannot have dimension $2$, 
so that this case is excluded. 

Let $e=0$ and assume first that $X'=\bbP(\calE)$, 
with $\calE$ the unique indecomposable rank $2$ vector bundle 
with invariant $e=0$ on an elliptic curve $\Gamma$. 
Then $X'$  has a section $E\cong \Gamma$ such that $\calO_E(E)=\calO_E$. 
We denote by $F$ the fibre class of the structure morphism $f: X'\longrightarrow \Gamma$.

The linear system $|C'|$ is of the form $|2E+F|$. 
We claim that $|2E+F|$ has dimension $2$ 
and has a base point scheme of length $2$ with support on $E$. 
Since $(C')^2 =4$, this will prove that this case cannot happen. 

 Indeed, let us  consider the exact sequence
$$
0\longrightarrow \calO_{X'}(F)\longrightarrow \calO_{X'}(E+F))\longrightarrow \calO_E(E+F) \longrightarrow 0.
$$
We have $h^0(X',\calO_{X'}(F))=1$ and $h^1(X',\calO_{X'}(F))=0$.
Since $\deg(\calO_E(E+F))=1$, we have 
$h^0(E, \calO_E(E+F))=1$ and $h^1(E, \calO_E(E+F))=0$.
As a consequence we have $h^0(X',  \calO_{X'}(E+F)))=2$ and 
$h^1(X',  \calO_{X'}(E+F)))=0$.

Look now at the exact sequence
$$
0\longrightarrow \calO_{X'}(E+F)\longrightarrow \calO_{X'}(2E+F))\longrightarrow \calO_E(2E+F) \longrightarrow 0.
$$
We have seen that $h^0(X',  \calO_{X'}(E+F)))=2$ and 
$h^1(X',  \calO_{X'}(E+F)))=0$. 
Since $\deg(\calO_E(2E+F))=1$, we have  $h^0(E, \calO_E(2E+F))=1$. 
Hence we deduce that  $h^0(X',  \calO_{X'}(2E+F)))=3$, 
i.e., $|2E+F|$ has dimension $2$. 
Moreover both $|2E+F|$ and $|E+F|$ have the same base point $p\in E$ 
cut out on $E$ by $F$. 
Moreover, $E+|E+F|$ is a pencil contained in $|2E+F|$ 
that has $p$ as a base point of multiplicity $2$. 
This implies that the base point scheme of $|2E+F|$ at $p$ has length $2$, as claimed. 

This shows that when $e=0$, $\calE$ cannot be indecomposable.

If in the above set--up $e=0$ and $\calE$ is decomposable, 
then we have at least two distinct sections $E$ ands $E'$ of $X'$ 
with $E^2=(E')^2=0$ and $E\cdot E'=0$. 
The linear system $|C'|$ is again of the form $|2E+F|$ 
and it intersects both $E$ and $E'$ in one point, 
so it has at least two distinct base points, 
and this implies that $|C|$ cannot determine a morphism of $\tilde X$ 
of degree $3$ to $\bbP^2$.

This finishes the analysis when $e=0$.

Finally, let $e=1$. 
In this case $X'=\bbP(\calO_\Gamma\oplus \calO_\Gamma(-p))$, 
with $p\in \Gamma$, 
and $X'$ has two distinct sections  $E$ and $E'$ 
with $E\cdot E'=0$, $E^2=-1$ and $(E')^2=1$. 

The linear system $|C'|$ is of the form $|2E+2F|$ and $E\cdot (2E+2F)=0$. 
So, if $\calO_E(2E+2F)$ is non--trivial, 
$E$ is a fixed component of $|C'|$ and we have a contradiction. 
Therefore $\calO_E(2E+2F)\cong \calO_E$.

Consider now the exact sequence
$$
0\longrightarrow \calO_{X'}(2F)\longrightarrow \calO_X(E+2F))\longrightarrow \calO_E(E+2F) \longrightarrow 0.
$$
We have $h^0(X',\calO_{X'}(2F))=2$ and $h^1(X',\calO_{X'}(2F))=0$.
Since $\deg(\calO_E(E+2F))=1$, we have 
$h^0(E, \calO_E(E+F))=1$ and $h^1(E, \calO_E(E+F))=0$.
As a consequence we have $h^0(X',  \calO_{X'}(E+2F)))=3$ and 
$h^1(X',  \calO_{X'}(E+2F)))=0$.

Look now at the exact sequence
$$
0\longrightarrow \calO_{X'}(E+2F)\longrightarrow \calO_{X'}(2E+2F))\longrightarrow \calO_E \longrightarrow 0.
$$
We have seen that $h^0(X',  \calO_{X'}(E+2F)))=3$ and 
$h^1(X',  \calO_{X'}(E+2F)))=0$. 
Hence we deduce that  $h^0(X',  \calO_{X'}(2E+2F)))=4$. 
Let $C'\in |2E+2F|$ that is smooth and irreducible of genus $2$. 
The linear series cut out by $|2E+2F|$ on $C'$ is a $g^2_4$ 
that is therefore base point free. 
Consider the morphism $\varphi_{|2E+2F|}: X'\longrightarrow \bbP^3$. 
This morphism cannot be birational, because, if it were, 
its image $Z$ would be a quartic surface 
with general plane sections of geometric genus $2$, 
so that $Z$ would have a line $t$ of double points. 
But then the planes through $t$ would cut out on $Z$ off $t$ a pencil of conics. 
These conics cannot be irreducible, 
because in that case $Z$, and therefore also $X'$, 
would be rational, a contradiction. 
This implies that the conics are all reducible, i.e., $Z$ is a scroll, 
and precisely an elliptic scroll. 
But then it is not possible that the general plane section of $Z$ has geometric genus $2$. 
So the image $Z$ of $\varphi_{|2E+2F|}$ is a quadric 
and  $\varphi_{|2E+2F|}: X'\longrightarrow Z$ is a double cover. 
But then this implies that $|C|$ cannot determine a morphism 
of $\tilde X$ of degree $3$ to $\bbP^2$.
  
This ends the analysis of the final case where $e=1$,
and finishes the proof of Claim \ref {cl:3}.
\end{proof}

So $C_\eta$ is not $2$--connected and we have a decomposition $C_\eta=A+B$, 
with $A$, $B$ effective and $A\cdot B=1$, 
so that $A,B$ are both $1$--connected by \cite [Lemma (A.4)]{CFM}. 
By \cite [Lemma (2.6)]{ML}, 
we may assume that either $A^2=-1$ or $A^2=0$. 
To finish the proof we will show that neither of these possibilities for $A^2$ can occur.

\begin{claim}\label{cl:4} The case  $A^2=-1$ cannot occur.
\end{claim}

\begin{proof} [Proof of Claim  \ref {cl:4}] 
If $A^2=-1$ then $C\cdot A=A^2+A\cdot B=0$ 
and $3=C^2=A^2+2A\cdot B+B^2=1+B^2$, so that $B^2=2$. 
Moreover $-1=K_{\tilde X}\cdot A+K_{\tilde X}\cdot B$. 
Since $C$ intersects positively all $(-1)$--divisors on $\tilde X$, 
$K_{\tilde X}\cdot A$ cannot be negative, so $K_{\tilde X}\cdot A>0$. 
We claim that $K_{\tilde X}\cdot A=1$ and $K_{\tilde X}\cdot B=-2$, 
so that $p_a(A)=p_a(B)=1$. 
Indeed we have $p_a(A)+p_a(B)=2$ and $p_a(B)\geq 1$ 
because otherwise $B$ would move in a $3$--dimensional family of rational curves, 
which is impossible. 
Hence $K_{\tilde X}\cdot B\geq -2$ which implies the assertion. 

Suppose first that $\tilde X$ is minimal. 
Then, as usual by now,  $\tilde X=\bbP(\calE)$, 
where $\calE$ is a normalized rank $2$ vector bundle on a smooth elliptic curve $\Gamma$, with invariant $e\geq -1$.  
Let us denote by $E$ a section with $E^2=-e$ 
and by $F$ a curve of the Albanese pencil of $\tilde X$. 
Then $K_{\tilde X}\equiv -2E-eF$. 
We have $C\equiv uE+vF$, with $u,v$ suitable integers and $u\geq 2$. 
Imposing $C^2=3$ we find $u(2v-ue)=3$, 
which forces $u=3$ and $v=(3e+1)/2$ (and therefore $e$ odd).
Imposing $C\cdot E\geq 0$, we find $3e\leq 1$, hence $e=-1$. 
Now let us put $A\equiv aE+\lambda F$, $B\equiv (3-a)E+\mu F$, 
with $a,\lambda, \mu$ suitable integers and $1\leq a\leq 2$. 

Suppose first $a=2$. 
Imposing $A\cdot B=1$, we find $2\mu+\lambda=-1$. 
Imposing $p_a(A)=1$, i.e., $A\cdot (A+K_{\tilde X})=0$, 
we find $\lambda =-1$, hence $\mu=0$. 
But then we would have $2=B^2=E^2=1$, a contradiction.

Suppose then $a=1$. 
Imposing $A\cdot B=1$, we find $2\lambda+\mu=-1$. 
Then $-1=A^2=(E+\lambda F)^2=1+2\lambda$, hence $\lambda=-1$. 
Therefore $\mu=1$ and we have $2=B^2=(2E+F)^2=8$, giving a contradiction again. 
The conclusion is that $\tilde X$ cannot be minimal.

Suppose next that $\tilde X$ is not minimal. 
Hence we have a blow--up $f: \tilde X\longrightarrow X'$, 
of a minimal elliptic ruled surface $X'$ at a point $y$. 
We let $C', A', B'$ as usual be the images of $C,A,B$ on $X'$ 
(with $A',B'$ both $1$--connected because $A$ and $B$ are) 
and we let $\gamma$ be the multiplicity of $C'$ at $y$. 
With the usual conventions, we can write $C'\equiv rE+sF$, 
with $r, s$ integers such that $r\geq 2$.
Imposing that $(C')^2=3+\gamma^2$, we get $r(2s-re)=3+\gamma^2$. 
Imposing that the geometric genus of $C'$ is $2$, 
we get the relation $\gamma^2+3=r(\gamma+1)$. 
Considering this as a degree $2$ equation in $\gamma$,
we see that the discriminant is $(r+6)(r-2)=(r+2)^2-16$.
We need this to be the square of a non--negative number $\delta$, 
hence we need $(r+2)^2-\delta^2=16$, i.e., $(r+2+\delta)(r+2-\delta)=16$. 
Notice that we cannot have $r+2-\delta=1$ and $r+2+\delta=16$ 
because, summing up these two relations, we get $2(r+2)=17$, 
which is impossible. So we have
$$
r+2+\delta=2^\alpha, \quad r+2-\delta=2^\beta
$$
so that $\alpha+\beta=4$ and $3\geq \alpha\geq \beta\geq 1$. 
For $\alpha=\beta=2$ we get $r=2, \delta=0, \gamma=1, s=e+1$ 
and for $\alpha=3, \beta=1$ we get $r=\delta=3$ and $\gamma$ is either $0$ or $3$. 
In the latter case $s= (3e+1)/2$ if $\gamma=0$ or $s=(3e+4)/2$ if $\gamma=3$.
When we impose that $C'\cdot E\geq 0$ we get 
$e\leq 1$ in the $r=2$ case and in the $r=\gamma=3$ case
and  $e\leq 0$  in the $r=3, \gamma=0$ case.
If $r=\gamma=3$, $e$ must be even in order that $s$ is an integer, hence $e=0$.
Similarly if $r=3$ and $\gamma=0$, $e$ must be odd, forcing $e= -1$.

Suppose we are in the case $r=2$, i.e., $C'\equiv 2E+(e+1)F$ and $-1\leq e\leq 1$.

First let $e=-1$. 
Then $C'\equiv 2E$ and $|C'|$ has dimension $2$ with no base points 
(see \cite [Thms. (1.17) and (1.18)] {CC}). 
This implies that the linear subsystem of $|C'|$ 
consisting of the curves in $|C'|$ containing the point $y$ has dimension $1$ 
and therefore $|C|$ cannot have dimension $2$, so that this case is excluded. 
 
Let $e=0$ and $\calE$ indecomposable so that $\calO_E(E)=\calO_E$. 
Then $C'\equiv 2E+F$ so that $C'\cdot E=1$ and $|C'|$ has a base point $z\in E$. 
 
First of all, let us consider the linear system $|E+F|$. 
We claim that $\dim (|E+F|)=1$ and that this linear system 
has the same base point $z$ on $E$. 
In fact, consider the exact sequence
 $$
 0\longrightarrow \calO_{X'}(F)\longrightarrow \calO_{X'}(E+F)
\longrightarrow \calO_{E}(E+F)\longrightarrow 0.
 $$
 As we saw already, $h^0( X',  \calO_{X'}(F))=1$ and $h^1( X',  \calO_{X'}(F))=0$, whereas $h^0( E,  \calO_{E}(E+F))=1$ and $h^1( E,  \calO_{E}(E+F))=0$, 
so that $h^0(X', \calO_{X'}(E+F))=2$ and $h^1(X', \calO_{X'}(E+F))=0$. 
 
Similarly, looking at the exact sequence 
 $$
 0\longrightarrow \calO_{X'}(E+F)\longrightarrow \calO_{X'}(2E+F)
\longrightarrow \calO_{E}(2E+F)\longrightarrow 0
 $$
 we see that $h^0(X', \calO_{X'}(2E+F))=3$ and $h^1(X', \calO_{X'}(2E+F))=0$. 
Hence $\dim (|2E+F|)=2$ and, as we said, $|2E+F|$ has a base point $z$ on $E$.
 
 Now notice that  $|2E+F|$ contains the linear subsystem $E+|E+F|$ of dimension $1$, 
that has $z$ as a base point which is singular for all curves of the system. 
This implies that $|2E+F|$ has a base point infinitely near to $z$. 
Hence $|C'|$ has two base points (i.e., $z$ and the infinitely near base point) 
and therefore it cannot define a degree $3$ rational map to $\bbP^2$. 
 
 Next let $e=0$ and $\calE$ decomposable 
so that 
$\calE\cong \calO_\Gamma\oplus \calO_\Gamma(\mathfrak e)$,
where $\mathfrak e\in {\rm Pic}^0(\Gamma)$. 
Note that we have here the section $E$ such that $\calO_E(E)=\mathfrak e$ 
and also a distinct section $E'$ such that $\calO_{E'}(E')=\mathfrak e^\vee$, 
with $E\cap E'=\emptyset$.  
Again $C'\equiv 2E+F$ so that $C'\cdot E=C'\cdot E'=1$ 
and $|C'|$ has two base points $z\in E$ and $z'\in E'$. 
Again $|C'|$ has two base points (i.e., $z$ and $z'$) 
and therefore it cannot define a degree $3$ rational map to $\bbP^2$.  

Let $e=1$, so that $\calE\cong \calO_\Gamma\oplus \calO_\Gamma(-p)$, 
where $p\in \Gamma$ and $\calO_E(E)=\calO_E(-p)$, 
where we identify $E$ with $\Gamma$ (see \cite [Thm. 2.12, p. 376]{har}). 
We have $C'\equiv 2(E+F)$. 
Let us fix the fibre $F_q$ of the Albanese pencil on $X'$ 
such that $F_q\cap E=\{q\}$, and consider the exact sequence
\begin{equation}\label{eq:worp}
0\longrightarrow \calO_{X'} (E+2F_q)\longrightarrow \calO_{X'} (2E+2F_q)
\longrightarrow \calO_{E} (2E+2F_q)\longrightarrow 0.
\end{equation}
So  
$\calO_E(E+F_q)=\calO_E(q-p)$ and
$\calO_E(2(E+F_q))=\calO_E(2q-2p)$. 
Moreover\\  $h^i(X', \calO_{X'} (E+2F_q))=0$ for $i=1,2$, 
because $E+2F_q-K_{X'}\equiv 3E+3F_q$ is big and nef;
hence the assertion follows by the Kawamata--Viehweg theorem. 
Note that $\chi(\calO_{X'} (E+2F_q))=3$.

If $\calO_E(q-p)$ is not of $2$--torsion on $E$,
the line bundle $\calO_{E} (2E+2F_q)$ is non--trivial, 
and therefore $E$ is a fixed component of $|2E+2F_q|$; 
then the curves in $|2E+2F_q|$ are all reducible 
and therefore also the curves in $|2E+2F_q|$ that contain $y$ are all reducible, 
which implies that $|C'|$ cannot be of the form $|2E+2F_q|$ as above. 

Hence is must be the case that $\calO_E(q-p)$ is of $2$--torsion on $E$. 
Since, as we saw, $h^1(X', \calO_{X'} (E+2F_q))=0$, 
from the sequence \eqref {eq:worp}, 
we have that $h^0( X', \calO_{X'} (2E+2F_q))=4$ in this case 
and $E$ is no longer a fixed component of $|2E+2F_q|$. 

We claim that $|2E+2F_q|$ has no base point. 
Let $x$ be a point of $E$  and assume first that $x\neq p$. 
Then look at the exact sequence
\begin{equation}\label{eq:rob}
0\longrightarrow \calO_{X'}(2E+2F_q-F_x)\longrightarrow \calO_{X'}(2E+2F_q) \longrightarrow \calO_{F_x}(2E+2F_q)\longrightarrow 0.
\end{equation}
We have $E\cdot (2E+2F_q-F_x)=-1$ so that $E$ splits off from $|2E+2F_q-F_x|$. 
Moreover $\calO_{E}(E+2F_q-F_x)\cong \calO_E(-p+2q-x)\cong \calO_E(p-x)$ 
which is non--trivial on $E$ 
so that $E$ splits off twice from $|2E+2F_q-F_x|$ 
and the residual system is $|2F_q-F_x|$ that has dimension $0$. 
So from \eqref {eq:rob} we see that the restriction map 
$H^0( X',\calO_{X'}(2E+2F_q)) \longrightarrow H^0(F_x,\calO_{F_x}(2E+2F_q))$ 
is surjective so that $|2E+2F_q|$ has no base point on $F_x$. 
We are left to check that $|2E+2F_q|$ has no base points on $F_p$. 
Consider the exact sequence 
\begin{equation}\label{eq:roba}
0\longrightarrow \calO_{X'}(F_p)\longrightarrow \calO_{X'}(E+F_p) \longrightarrow \calO_{E}(E+F_p)\cong \calO_E\longrightarrow 0.
\end{equation}
As usual $h^i(X', \calO_{X'}(F_p))=0$, for $i=1,2$, 
whereas $h^0(X', \calO_{X'}(F_p))=1$. 
Form \eqref {eq:roba}, we see that $|E+F_p|$ has dimension $1$, 
it does not have $E$ as a fixed component, 
its general curve is smooth and irreducible 
and $h^1(X', \calO_{X'}(E+F_p))=1$. 
The pencil $|E+F_p|$, that has self--intersection $1$, 
has a base point on $F_p$. 
Indeed, look at the exact sequence
\begin{equation*}\label{eq:robb}
0\longrightarrow \calO_{X'}(E)\longrightarrow \calO_{X'}(E+F_p) 
\longrightarrow \calO_{F_p}(E+F_p)\cong \calO_{\bbP^1}(1)\longrightarrow 0.
\end{equation*}
Since $h^0(X', calO_{X'}(E+F_p))=2$ and $h^0(X', \calO_{X'}(E))=1$,
we see that the restriction map 
$$
H^0(X', \calO_{X'}(E+F_p))\cong \bbC^2\longrightarrow H^0(F_p, H^0(F_p, \calO_{F_p}(E+F_p))\cong \bbC^2
$$
is not surjective, 
and this proves that $|E+F_p|$ has a base point $u$ on $F_p$. 
Notice that $u$ is not the intersection point of $F_p$ with $E$, 
because $E\cdot (E+F_p)=0$. 
Since $2E+2F_q\sim 2E+2F_p$, because $2p\sim 2q$ on $E$, 
the point $u$ can be the only base point of $|2(E+F_q)|$ on $F_p$. 
But now we see that $u$ is not a base point for $|2(E+F_q)|$.
 Let in fact $D\in |E+F_p|$ be a general curve, 
that is smooth, irreducible and passes through $u$. 
Consider the exact sequence 
\begin{equation}\label{eq:robc}
0\longrightarrow \calO_{X'}(E+F_p)\longrightarrow \calO_{X'}(2E+2F_q) \longrightarrow \calO_{D}(2E+2F_q)\longrightarrow 0.
\end{equation}
We have $h^0(X',  \calO_{X'}(E+F_p))=2$, $h^1(X',  \calO_{X'}(E+F_p))=1$, 
$h^0(X',  \calO_{X'}(2E+2F_q))=4$, 
$h^1(X',  \calO_{X'}(2E+2F_q))=1$ 
and $h^0(D,  \calO_{D}(2E+2F_q))=2$, $h^1(D,  \calO_{D}(2E+2F_q))=0$. 
Hence from \eqref {eq:robc} we deduce that the restriction map 
$$
H^0(X', \calO_{X'}(2E+2F_q))\longrightarrow H^0(D, H^0(D, \calO_{D}(2E+2F_q))
$$
is surjective, and this proves that $|2E+2F_q|$ has no base point on $D$, 
and in particular $u$ is not a base point for $|2E+2F_q|$. 

In conclusion $\varphi_{|2E+2F_q|}: X'\longrightarrow \Sigma$ is a morphism 
that maps $X'$ to a non--degenerate surface $\Sigma$ in $\bbP^3$. 
This map cannot be birational. 
Indeed, if is was birational, $\Sigma$ would be a quartic surface 
with general plane sections of genus $2$, 
so $\Sigma$ would have a double line $r$. 
This can be excluded with an argument that we already made above.
So $\varphi_{|2E+2F_q|}$ has degree $2$ and $\Sigma$ is a quadric. 
But then the sublinear system of  curves in $|2E+2F_q|$ containing the point $y$ 
cannot define a triple cover of $\bbP^2$. 

This finishes the case with $r=2$, $s=e+1$, and $\gamma=1$.
Let us turn our attention to the other two cases, which have $r=3$.

If $\gamma=0$, then we have $s=-1$ and $e=-1$; 
let us deal with this case next.
Here $C\equiv 3E-F$ and 
$X'$ is the symmetric product of a genus one curve $\Gamma$.
By \cite[Theorem 1.17]{CC}, we have that $H^0(X',C)=2$, 
and this is a contradiction to giving us the triple cover map.

Finally we have the case $r=3$, $\gamma =3$, $e=0$, and $s=2$.
In this case we have that $C$ has a triple point $p$, 
and we make an elementary transformation of the elliptic ruled surface $X'$ at $p$.
We note that $p \notin E$, because $C\cdot E =2$;
for the same reason, no other section (if it exists) of self-intersection zero
can contain $p$.
We also note that every section of the ruled surface has even self-intersection.
After making the elementary transformation, 
we obtain a minimal elliptic ruled surface $X''$ 
and $E$ transforms to a section $E''$ with self-intersection one.
Therefore $E''$ is a section with minimal self-intersection on $X''$;
hence $X''$ has invariant $e=-1$.

The proper transform $C''$ of the curve $C$ has numerical class $3E''-F$,
which has self-intersection $3$.
Again by \cite[Theorem 1.17]{CC}, 
we have that $H^0(X'',C'')=2$, 
and this is a contradiction to giving us the triple cover map.
\end{proof}

\begin{claim}\label{cl:5} 
The case  $A^2=0$ cannot occur.
\end{claim}

\begin{proof} [Proof of Claim \ref {cl:5}] 
If $A^2=0$, then  $C\cdot A=A^2+A\cdot B=1$ 
and $3=C^2=A^2+2A\cdot B+B^2=2+B^2$, 
so that $B^2=1$. 
Again $-1=K_{\tilde X}\cdot A+K_{\tilde X}\cdot B$. 
If $K_{\tilde X}\cdot A<0$, 
we would have that  $p_a(A)=0$ and $A$ moves in the Albanese pencil. 
This is impossible because $C\cdot A=1$ and then $C$ would have genus $1$, 
instead of genus $2$. 
So $K_{\tilde X}\cdot A\geq 0$. 
Similarly, $p_a(B)$ cannot be $0$, because, since $B^2=1$,
$B$ would move in a linear net of rational curves, 
so that $K_{\tilde X}\cdot B\geq -1$, 
and this implies that $K_{\tilde X}\cdot A= 0$ and $K_{\tilde X}\cdot B=-1$, 
so that $p_a(A)=p_a(B)=1$. 
Since $C\cdot A=1$, 
there is a $g^1_1$ on $A$ corresponding to a line bundle $\calN$. 
Then we apply again \cite [Proposition (A.5)]{CFM} 
and deduce that there is a decomposition $A=A_1+A_2$, 
with $A_1, A_2$ effective divisors,  $A_1\cdot A_2=1$ 
and $\calN_{|A_1}=\calO_{A_1}(A_2)$, 
so that $\deg(\calN_{|A_1})=1$, and $\deg(\calN_{|A_2})=0$. 

Consider the decomposition $C=A_2+(A_1+B)$. 
We have $1\leq A_2\cdot (A_1+B)=1+ A_2\cdot B$, 
hence $A_2\cdot B\geq 0$. 
If the equality holds, we have 
$A_2\cdot (A_1+B)=1$.
Note that $C\cdot A_2 =0$, and then $A_2^2 <0$ by the index theorem.
Moreover $0=A^2 = A_1^2+A_2^2+2$, hence $A_1^2+A_2^2 = -2$.
Since $A_2^2 <0$, we have $A_1^2 \geq -1$.
Then $(A_1+B)^2 = A_1^2+3 \geq 2$, 
which implies (by \cite [Lemma (2.6)]{ML})
that $A_2^2 = -1$, which already has been excluded by Claim \ref {cl:4}. 

Hence we may assume that $A_2\cdot B\geq 1$. 
Note that $p_a(A_2)>0$. 
(Indeed, if $p_a(A_2)=0$, since $A_2$ is $1$--connected by \cite [Lemma (A.4)]{CFM}, 
$A_2$ would be contained in a fibre of the Albanese pencil. 
Since $C\cdot A_2=0$, $A_2$ cannot be a fibre of the Albanese pencil, 
and it also cannot be a $(-1)$--curve. 
But by Claim \ref {cl:due} there are no other possibilities for $A_2$, 
so $p_a(A_2)>0$. )
Then we have $p_a(A_2+B)\geq 2$ and 
$$
2=p_a(C)=p_a(A_1+(A_2+B))\geq 2+A_1\cdot (A_2+B)-1=A_1\cdot (A_2+B)+1,
$$
so that $A_1\cdot (A_2+B)\leq 1$. 
On the other hand, since $C$ is $1$--connected, 
we have also $A_1\cdot (A_2+B)\geq 1$, 
so that $A_1\cdot (A_2+B)=1$, which implies that 
$A_2\cdot B=1$, $A_1\cdot B=0$ and $p_a(A_2)=1, p_a(A_1)=0$. 
Then $A_1$ is contained in a fibre of the Albanese pencil 
and cannot be a fibre of the Albanese pencil because $C\cdot A_1=1$. 
Then by Claim \ref {cl:due}, $A_1$ is a $(-1)$--curve, 
and then $A_2^2 = -1$ as well.  
We can then consider the morphism $f: \tilde X\longrightarrow X'$ 
that contracts $A_1$ to a smooth point $y$. 
Set $C'=f_*(C)$, $A'=f_*(A)=f_*(A_2)$, $B'=f_*(B)$. 
We have $(C')^2=4$, $(A')^2=0$, $(B')^2=1$, $C'\sim A'+B'$, with $A'\cdot B'=1$.
Then $4= (C')^2 = (A')^2+(B')^2+2(A'\cdot B') = 3$, a contradiction.
%
%
%
%
\end{proof} 

As noted above, the proofs for the two Claims \ref{cl:4} and \ref{cl:5}
finish the proof of Proposition \ref{prop:deg=8notq=1}.
\end{proof}

\subsection{The rational case}

The only case left to consider is when the triple cover is a rational surface.

\begin{proposition}\label{prop:plane} 
Let $\pi: X\longrightarrow \bbP^2$ be a triple plane 
with branch curve of degree $8$ and $X$ rational. 
Let $\phi: \tilde X\longrightarrow X$ be the minimal desingularization of $X$ 
and let $\tilde {\calL}$ be the pull back via $\phi$ of $\calL$. 
Then  $\tilde X$ is the blow--up of $\bbP^2$ at $10$ points and, 
up to a Cremona map, $\tilde {\calL}$ is the linear system $|4;2, 1^9|$. 
\end{proposition}

\begin{proof} 
As usual we abuse notation and denote also by $C$
the general curve in $\tilde {\calL}$, 
which is a smooth curve of genus $2$ with $C^2=3$;
hence $K_{\tilde X}\cdot C=-1$. 

The adjoint linear system $|K_{\tilde X}+C|$ has dimension $1$ 
and so we can write it as $|K_{\tilde X}+C|=\Phi+|F|$, 
where $\Phi$ is the fixed part and $|F|$ the movable part 
that is a linear pencil whose general curve is irreducible. 
We have $F\cdot C=2$ and $\Phi\cdot C=0$ 
because $|K_{\tilde X}+C|$ cuts out on the general curve $C \in \tilde {\calL}$ 
the complete canonical series, that is base point free. 

We claim that $|F|$ is a pencil of rational curves. 
If fact, if $F\in |F|$ is a general curve, 
$|C|$ cuts out a base point free $g_2^r$ on it, 
so that $r\geq 1$. 
However, $r=1$ is not possible, 
because otherwise $\phi_{|C|}$ would have degree $2$ onto it image, 
which is impossible. 
So $r=2$ and $F$ is rational (and smooth). 
Hence $F^2=0$ and $K_{\tilde X}\cdot F=-2$. 

Then we have 
$$
K_{\tilde X}^2-1=K_{\tilde X}\cdot (K_{\tilde X}+C)
=K_{\tilde X}\cdot F+K_{\tilde X}\cdot \Phi=-2+K_{\tilde X}\cdot \Phi
$$
hence $K_{\tilde X}\cdot \Phi=K_{\tilde X}^2+1$. 

Now look at the exact sequence
$$
0\longrightarrow \calO_{\tilde X}(\Phi-K_{\tilde X})\longrightarrow \calO_{\tilde X}(C)\longrightarrow \calO_{F}(C)\longrightarrow 0.
$$
As we saw, the map
$H^0(\tilde X, \calO_{\tilde X}(C))\longrightarrow H^0(F, \calO_{F}(C))\cong \bbC^3$ 
is surjective. 
On the other hand $h^0(\tilde X, \calO_{\tilde X}(C))=3$, 
because if we set $r+1:=h^0(\tilde X, \calO_{\tilde X}(C))$, 
the linear system $|C|$  cuts out on the general curve $C\in |C|$ a $g^{r-1}_3$ 
and of course $r\leq 2$. 
So  $H^0(\tilde X, \calO_{\tilde X}(C))\longrightarrow H^0(F, \calO_{F}(C))\cong \bbC^3$
is an isomorphism. 
This implies that  $h^0(\tilde X, \calO_{\tilde X}(\Phi-K_{\tilde X}))=0$.
By Riemann--Roch, it follows that $h^1(\tilde X, \calO_{\tilde X}(C))=0$;
hence also $h^1(\tilde X, \calO_{\tilde X}(\Phi-K_{\tilde X}))=0$. 
Since 
$h^2(\tilde X, \calO_{\tilde X}(\Phi-K_{\tilde X}))
=h^0(\tilde X, \calO_{\tilde X}(2K_{\tilde X}-\Phi))=0$, 
we have $\chi(\calO_{\tilde X}(\Phi-K_{\tilde X}))=0$ 
which is equivalent to $\Phi^2- K_{\tilde X}^2=1$.
As we have also $K_{\tilde X}\cdot \Phi=K_{\tilde X}^2+1$, 
it follows that $\Phi^2=K_{\tilde X}\cdot \Phi$. 

We claim now that $\Phi=0$. 
If not, since $C\cdot \Phi=0$, 
it is the case that $C\cdot D =0$ for all components of $\Phi$;
therefore $D^2<0$ for all such components.
Since $K_{\tilde X}\cdot \Phi<0$
(because $\Phi^2<0$ and $\Phi^2=K_{\tilde X}\cdot \Phi$), 
then there would be an irreducible component $D$ of $\Phi$ 
such that $D^2<0$ and $K_{\tilde X}\cdot D<0$. 
So $D$ would be a $(-1)$--curve such that $C\cdot D=0$, which is impossible. 

In conclusion $K_{\tilde X}^2=-1$ and $|K_{\tilde X}+C|=|F|$. 
So, up to a Cremona transformation, 
we may assume that $|F|$ is the linear system $|1;1|$ 
and therefore $\tilde {\calL}$ is a system $|4;2,1^9|$ as claimed. 
\end{proof}

If $X$ is the blow--up of the plane at ten general points, and 
$\pi: X\longrightarrow \bbP^2$ is the otp given by the map determined by the linear system $|4;2,1^9|$, then the Tschirnhausen vector bundle is 
$\calO_{\mathbb P^2}(-2)^{\oplus 2}$ 
and its branch curve has degree $8$;
it is irreducible and it has in general $12$ ordinary cusps 
(see \cite [Lemma 10.1 and Table 10.5]{M}). 
However, if the ten blown--up points are in special position 
so that there is some curve contracted by the map 
determined by the linear system $|4;2,1^9|$, 
the branch curve may acquire more singular points. 
We will not dwell on the description of all special cases here.

From the viewpoint of the structure constants for the triple cover
as in \eqref{triplecoverstructureconstants} below,
each of the polynomials $a,b,c,d$ has degree $2$;
the discriminant $D$ defining the branch curve is a homogeneous quartic in these polynomials, which gives the degree eight description as well.
We can also check the count of parameters for this case too.
The ten points in the plane have $20-8=12$ moduli.
The four quadratics have $24$ coefficients; there are $8$ automorphisms of $\bbP^2$,
and $4$ for the Tschirnhausen bundle, also giving $12$ moduli.

\begin{example}\label{ex:klo} 
There are various geometric ways 
to look at rational triple planes with branch curve of degree $8$. 
For example, there are quintic surfaces $S$ in $\bbP^3$ 
with a singular curve $Y$ that is an elliptic quartic curve (see \cite [\S 32]{Co}). 
They are rational surfaces with sectional genus $2$. 
The projection of such a surface $S$ from a general point $x\in Y$, 
determines an otp $\pi: X\longrightarrow \bbP^2$ of the required type. 
By taking into account the plane representation of the surfaces $S$ 
(as explained in \cite [\S 32]{Co}), 
the reader will see that these triple planes 
are of the type described in Proposition \ref {prop:plane}. 
We do not dwell on this here. 
\end{example}

\section{Branch curve of degree 10}\label{sec:10}

Consider next a triple plane $\pi: X\longrightarrow \bbP^2$ 
such that its branch curve has degree $10$. 
Such a triple plane is certainly ordinary.
Then the general curve $C\in \calL$ is smooth of genus  $g=3$ with $C^2=3$. 

Let $\phi: \tilde X\longrightarrow X$ be the minimal desingularization of $X$ 
and let $\tilde {\calL}$ be the pull back via $\phi$ of $\calL$. 
As usual, we abuse notation and denote by $C$ 
also the general curve in $\tilde {\calL}$, 
that is a smooth curve of genus $3$ with $C^2=3$, 
and hence $K_{\tilde X}\cdot C=1$. 
Let us denote by $p_g$ and $q$ the geometric genus and the irregularity of $\tilde X$. 
Since $K_{\tilde X}\cdot C=1$, 
it is clear that $p_g\leq 1$. 
Since $\tilde X$ is covered by curves of genus $3$, one has $q\leq 3$. 

\subsection{Exclusion of  the case $p_g=1$, $q>0$} 

\begin{proposition}\label{prop:no} 
There is no triple plane $\pi: X\longrightarrow \bbP^2$ 
with branch curve of degree $10$ with $q\geq p_g=1$.
\end{proposition}

\begin{proof}  
Since $p_g=1$, the Kodaira dimension of $X$ is at least $0$. 
By the index theorem, since $C^2=3$ and $K_{\tilde X}\cdot C=1$, 
one has $K_{\tilde X}^2\leq 0$. 

Suppose first that $q=1$ and that $\tilde X$ is minimal. 
Hence $K_{\tilde X}^2= 0$ and, by the classification, 
$\tilde X$ is a properly elliptic surface (i.e., with Kodaira dimension equal to $1$). 
Then the unique canonical curve is a fibre of the elliptic fibration. 
But then $K_{\tilde X}\cdot C=1$ is clearly a contradiction. 

Suppose again that $q=1$ and that $\tilde X$ is not minimal, 
so there is some $(-1)$--curve on $\tilde X$. 
Since any $(-1)$--curve is a component of the unique canonical curve 
and since $|C|$ intersects positively any $(-1)$--curve, 
we conclude that there is a unique $(-1)$--curve $E$ on $\tilde X$ 
and that $C\cdot E=1$. 
Let us consider the morphism $f: \tilde X \longrightarrow X'$ 
that contracts $E$ to a point $x$. 
Let us set $C'=f_*(C)$. 
Then $(C')^2=4$, $C'\cdot K_{X'}=0$ and again by the index theorem,
one has $K_{X'}^2\leq 0$. 
But now $X'$ is minimal, 
hence $K_{X'}^2=0$ and again $X'$ is a properly elliptic surface. 
But $C'\cdot K_{X'}=0$ implies that $C'$ is composed by fibres of the elliptic fibration,
which is impossible since $(C')^2=4$. 
So the case $q=1$ is excluded. 

Assume next $2\leq q\leq 3$. 
By the classification, the only possible case is that $q=2$ 
and $\tilde X$ is birational to an abelian surface. 
Since $C\cdot K_{\tilde X}=1$, 
$\tilde X$ is not minimal. 
Since any $(-1)$--curve is a component of the unique canonical curve on $\tilde X$ 
and since $|C|$ intersects positively any $(-1)$--curve, 
we conclude that there is a unique $(-1)$--curve $E$ on $\tilde X$ 
and that $C\cdot E=1$. 
Let us consider the morphism $f: \tilde X \longrightarrow X'$ 
that contracts $E$ to a point $x$, so that $X'$ is an abelian surface. 
Let us set $C'=f_*(C)$. 
Then $(C')^2=4$ and therefore $C'$ determines a polarization of type $(1,2)$ on $X'$, 
so that $\dim (|C'|)=1$, which excludes this case.  
\end{proof}

\subsection{The case $q=0$, $p_g=1$}  
The first example that comes to mind of such a triple plane
is the projection of a quartic surface in $\bbP^3$
from a point on the surface.
In fact this is the only such example, assuming canonical singularities.

\begin{proposition}\label{prop:4ic} 
If $q=0$, $p_g=1$ then 
there is an irreducible quartic surface $S\subset \bbP^3$ with canonical singularities, 
and there is a smooth point $p\in S$ 
such that $X$ is the blow--up of $S$ at $p$ 
and the triple plane $\pi: X\longrightarrow \bbP^2$ 
corresponds to the projection of $S$ to a plane from $p$.
\end{proposition}

\begin{proof} 
Consider the adjoint linear system $|K_{\tilde X}+C|$, 
that we can write as $|K_{\tilde X}+C|=|F|+\Phi$, 
with $\Phi$ the fixed part and $|F|$ the movable part. 
Note that $|F|$ cannot be composed with a pencil, 
otherwise the curves in $|C|$ would be hyperelliptic, 
which is not possible because they have a base point free $g^1_3$. 
So the general curve in $|F|$ is irreducible.

One has $C\cdot \Phi=0$ and $F\cdot C=4$. 
Moreover, since $h^i( \tilde X, \calO_{\tilde X}(K_{\tilde X}+C))=0$ for $1\leq i\leq 2$, 
by Riemann--Roch we have $\dim (|F|)=\dim (|K_{\tilde X}+C|)=3$. 
By looking at the exact sequence 
$$
0\longrightarrow \calO_{\tilde X}(F-C)\longrightarrow \calO_{\tilde X}(F)
\longrightarrow \omega_C\longrightarrow 0,
$$
we see that $h^0(\tilde X, \calO_{\tilde X}(F-C))>0$, hence $D:=F-C$ is effective.  
Since $C\cdot D=C\cdot F-C^2=1$, by the index theorem we deduce that $D^2\leq 0$. 
So $0\geq D^2=F^2+C^2-2F\cdot C=F^2-5$ so that $F^2\leq 5$. 
On the other hand $F^2=F\cdot C+F\cdot D\geq 4$. 
So either $F^2=4$, $D^2=-1$, $F\cdot D=0$ or $F^2=5$, $D^2=0$, $F\cdot D=1$. 
 
Note now that $|F|$  contains the subsystem of codimension $1$ given by $D+|C|$. 
Since $C$ has no base points, then the base points of $|F|$, if any, can be only on $D$. 
So if $F^2=4$ then $|F|$ has no base points, 
and if $F^2=5$ there can be at most one base point of $|F|$ on $D$. 

Let us first consider the case $F^2=4$. 
Then the morphism $\varphi_{|F|}: \tilde X\longrightarrow  S\subset \bbP^3$ 
mapping $\tilde X$ to a surface $S$, is birational onto its image, 
because otherwise, since $F^2=4$, it would determine a $2:1$ map onto a quadric, 
which is not the case because $|C|$, that is contained in $|F|$, 
determines a $3:1$ map to the plane. 

So $S$ is a quartic surface in $\bbP^3$. 
Since $S$ has geometric genus $1$, it has only canonical singularities. 
The map $\varphi_{|F|}$ contracts $D$ to a point $p\in S$ 
and the projection of $S$ from $p$ to a plane determines exactly the original triple plane. Hence $p$ is a simple point of $S$ and the assertion follows. 

Assume next that $F^2=5$ and that $|F|$ has a base point $p$ on $D$.  
For the same reasons as above, 
the map $\varphi_{|F|}: \tilde X\dasharrow  S\subset \bbP^3$ 
mapping $\tilde X$ to a surface $S$, is birational onto its image. 
Then $S$ is a quartic surface in $\bbP^3$, that, as above, has only canonical singularities. 
The map $\varphi_{|F|}$ contracts $D$ to a point $p\in S$ 
and the projection of $S$ from $p$ to a plane determines exactly the original triple plane. But then we are in the situation described before and we are not in the case $F^2=5$.

Similarly, if $F^2=5$ and  $|F|$ has no base points, 
then the morphism $\varphi_{|F|}: \tilde X\longrightarrow  S\subset \bbP^3$ 
mapping $\tilde X$ to a surface $S$, is birational onto its image 
and $S$ is a quintic surface in $\bbP^3$. 
The morphism $\varphi_{|F|}$ maps $D$ to a line $r$ 
and $|C|$ would be the pull--back on $\tilde X$ 
of the linear system of plane sections of $S$ through the line $r$. 
But then $|C|$ would have dimension $1$ and not $2$, which is a contradiction. 

So the case $F^2=5$ cannot occur, and this finishes the proof of the assertion. 
\end{proof}

We notice that if, in the above setup, $S$ is smooth and $p\in S$ is a general point, 
then the triple plane obtained by projecting $S$ from $p$ 
has Tschirnhausen bundle isomorphic to 
$\calO_{\bbP^2}(-2)\otimes \calO_{\bbP^2}(-3)$ 
and the branch curve has degree $10$ with, in general, exactly $18$ cusps 
(see again \cite [Lemma 10.1 and Table 10.5]{M}). 
If the surface $S$ is singular or $p$ is a special (simple) point, 
the branch curve may acquire more singular points. 
We will not dwell on the description of all special cases here. 

\begin{remark} \label{rem:mod} 
We  remark that the four structure constants $a,b,c,d$ as in \eqref{triplecoverstructureconstants} below 
have degrees $2,1,4,3$ respectively.
These together have $6,3,15,10$ coefficients respectively
(as forms in the homogeneous coordinates of the plane)
and therefore give $6+3+15+10-1 = 33$ dimensions of (projective) parameters
for the section of the relevant bundle for the triple cover construction as in \cite{M}.
Subtracting the $8$ dimensions for the automorphisms of the plane,
and the $4$ dimensions for the automorphisms of the Tschirnhausen bundle,
we obtain $21$ dimensions for this family of triple covers.
This is consonant with the $19$ dimensions for the parameters of the $K3$ quartic surface,
and two additional parameters for the point on the surface that is the point of projection.
\end{remark}

\subsection{The case $q>p_g=0$}\label{ssec:opla}

\begin{proposition}\label{prop:plop} 
If $\pi: X\longrightarrow \bbP^2$ is a triple plane with branch curve $B$ of degree $10$ 
with $q> p_g=0$, then the only possibilities are either $q=1$ or  $q=3$.
 
In the case $q=1$ then either $X$ is properly elliptic and its minimal desingularization is minimal or $X$ has Kodaira dimension $-\infty$. 

If $q=3$ then  the Kodaira dimension of $X$ is $-\infty$ and $B$ consists of $10$ lines in a pencil.
\end{proposition}

\begin{proof} 
Suppose first $q\geq 2$. 
Then $\chi(\calO_X)\leq -1$ and therefore the Kodaira dimension of $X$ is $-\infty$. 
Let, as usual $\phi: \tilde X \longrightarrow X$ be the minimal desingularization of $X$. 
Then $\tilde X$ is swept out by curves of genus $3$, hence $q\leq 3$. 
Let $f: \tilde X\longrightarrow \Gamma$ be the irrational pencil of curves $F$ of genus $0$ on $\tilde X$. 

If $q=3$, then the curves $C$ are unisecant to the curves of the pencil, 
i.e., $C\cdot F=1$. 
Since $C$ intersects any $(-1)$--curves on $\tilde X$, 
then $\tilde X$ must be minimal, i.e., all fibres of the irrational pencil 
$f: \tilde X\longrightarrow \Gamma$ are smooth and rational. 
Now, by arguing exactly as in the proof of Proposition \ref {prop:4}, 
we see that the irrational pencil $f: \tilde X\longrightarrow \Gamma$ 
is composed with the triple plane map, 
and the branch curve is composed of $10$ lines of a pencil. 

If $q=2$, then $f: \tilde X\longrightarrow \Gamma$ 
induces a morphism $g: C\longrightarrow \Gamma$, 
which is necessarily an \'etale double cover. 
Then the curves $C$ are bisecant 
to the curves of the irrational pencil $f: \tilde X\longrightarrow \Gamma$, 
i.e., $C\cdot F=2$ 
and the fibres of $f$ are mapped to the conics of a family $\calF$.

By Corollary \ref {cor:pomp}, either $\pi: X\longrightarrow \bbP^2$ 
is Cremona equivalent to an exceptional example, 
or the branch curve of $\pi: X\longrightarrow \bbP^2$ consists of $10$ lines in a pencil. 
So, to finish the proof, we have to exclude the former case. 
In that case in fact we have a diagram like  \eqref {eq:ex} 
(where $\Gamma$ is a smooth elliptic curve 
and $\Gamma[2]\longrightarrow \bbP^2$ is the basic exceptional example) 
and the rational map $\phi: \bbP^2 \dasharrow \bbP^2$ 
maps the conics of the family $\calF$ 
to the tangent lines to the branch curve of the basic exceptional example. 
This implies that $\phi$ is a quadratic transformation 
and therefore the branch curve of $\pi: X\longrightarrow \bbP^2$ 
would be easily seen to have degree $12$, a contradiction. 

Suppose finally that  $q=1$.  
Assume  that the Kodaira dimension of $X$ is not $-\infty$.  
By the index theorem, we have $K_{\tilde X}^2\leq 0$. 
Since $\chi(\calO_X)=0$, then $\tilde X$ is either bielliptic 
(so with Kodaira dimension $0$) or properly elliptic (with Kodaira dimension $1$).  
In any case there is an elliptic fibration $f: \tilde X \longrightarrow G$ 
(and $G$ is a smooth irreducible curve of genus $\gamma\leq 1$). 
Let $F$ be a general fibre of the elliptic fibration. 
Let us set $m=C\cdot F$ and note that $m\geq 2$.

Suppose that $\tilde X$ is not minimal. 
Then there is a non--negative rational number $t$ such that  $K_{\tilde X}\equiv tF+E$, 
where $E$ is the sum of a number $s$ of $(-1)$--divisors. 
Since $C$ intersects positively all $(-1)$--divisors, 
we have $1=C\cdot K_{\tilde X}=C\cdot E+tm\geq s+2t\geq s$. 
From this we deduce that $t=0$ and $s=1$. 
This implies that $E$ is a unique $(-1)$--curve and that $C\cdot E=1$. 
Let $f: \tilde X\longrightarrow X'$ be the blow--up of a point $p\in X'$, 
contraction of the $(-1)$--curve $E$, 
so that $X'$ is minimal and, by the classification  of surfaces, is bielliptic. 
Let $C'=f_*(C)$. Then $(C')^2=4$, $C'$ is big and nef and $h^i(X', \calO_{X'}(C'))=0$ 
by the Kawamata--Vieweg theorem,  
because $C'-K_{X'}$ is big and nef. 
So, by Riemann--Roch we have $h^0(X', \calO_{X'}(C'))=\chi(\calO_{X'}(C'))=2$, 
which gives a contradiction to $\dim (|C|)=2$. 
So in conclusion $t>0$, and $\tilde X$ is properly elliptic and minimal. 
\end{proof}

\begin{example}\label{ex:dol} An easy example of the $q=3$ case is given by the projection of a cone $S$ over a smooth plane quartic from a point on it.\end{example}

\begin{remark}\label{rem:kl} 
It would be nice to understand whether, in Proposition \ref {prop:plop}, 
the case $q=1$ can really occur, and to classify those cases. 

In any event, we can say a bit more about the case 
$q=1$ and $\tilde X$ properly elliptic and minimal. 
First of all from $C\cdot K_{\tilde X}=1$,  $K_{\tilde X}\equiv tF$, and $C\cdot F=m$, 
we deduce that $t=1/m$ hence $mK_{\tilde X}\equiv F$. 

Recall now the canonical bundle formula for $f: \tilde X \longrightarrow G$, 
which implies that
$$
K_{\tilde X}\equiv f^*(K_G)+\sum_{i=1}^h(m_i-1)F_i,
$$
where $m_iF_i$, for $1\leq i\leq h$, are the multiple fibres of $f$ 
(see \cite [Thm. 9.18]{C}). 
From this we deduce that
$$
\Big ( m\big (\sum_{i=1}^h \frac  {m_i-1}{m_i}\big)-1\Big )K_{\tilde X}
\equiv f^*(-K_{G}).
$$
Using this we see that if $G$ is elliptic, 
then we must have 
$$
m\big (\sum_{i=1}^h \frac  {m_i-1}{m_i}\big) =1;
$$
if not, then $G$ is rational, and we have
$$
\Big ( m\big (\sum_{i=1}^h \frac  {m_i-1}{m_i}\big)-1\Big )K_{\tilde X}\equiv 2F
$$
and therefore if we set $x = \sum_{i=1}^h (m_i-1)/m_i$, we have $m(x-2)=1$.
In both cases we obtain some numerical constraints on the parameters,
which we will spare ourselves and the reader to analyze further.
\end{remark}

\subsection{The case $q=p_g=0$}

\begin{proposition}\label{prop:plop2} 
If $\pi: X\longrightarrow \bbP^2$ is a triple plane with branch curve $B$ of degree $10$ 
with $q= p_g=0$, then $X$ is rational and only the following cases are possible:\\
\begin{inparaenum}[(i)]
\item $\tilde X$  is the blow up of $\bbP^2$ at $11$ points and, up to Cremona maps,  
$\tilde {\calL}$ is the linear system $|9;3^8,2, 1^2|$;\\
\item  $\tilde X$ is the blow up of $\bbP^2$ at $11$ points and, up to Cremona maps,  
$\tilde {\calL}$ is the linear system $|7;3,2^9, 1|$;\\
\item $\tilde X$ is the blow up of $\bbP^2$ at $12$ points and, up to Cremona maps,  
$\tilde {\calL}$ is the linear system $|6;2^7, 1^5|$;\\
\item $\tilde X$ is the blow up of $\bbP^2$ at $13$ points and, up to Cremona maps,  
$\tilde {\calL}$ is the linear system $|4;1^{13}|$.
\end{inparaenum}
\end{proposition}

\begin{proof} 
First of all we claim that $h^0(\tilde X, \calO_{\tilde X}(C))=3$ and 
$h^1(\tilde X, \calO_{\tilde X}(C))=1$.

In fact, consider the exact sequence
$$
0\longrightarrow \calO_{\tilde X}\longrightarrow \calO_{\tilde X}(C)
\longrightarrow \calO_{C}(C)\longrightarrow 0.
$$ 
We have $h^0(\tilde X, \calO_{\tilde X})=1$ and $h^i(\tilde X, \calO_{\tilde X})=0$ 
for $1\leq i\leq 2$. Hence 
$$
2\leq h^0(\tilde X, \calO_{\tilde X}(C))-1=h^0(C, \calO_{C}(C))\leq 2
$$
and so $h^0(\tilde X, \calO_{\tilde X}(C))=3$ and $h^0(C, \calO_{C}(C))=2$. 
Therefore $h^1(C, \calO_{C}(C))=1$ hence $h^1(\tilde X, \calO_{\tilde X}(C))=1$. 

Consider now the adjoint system $|K_{\tilde X}+C|=|F|+\Phi$, 
where $\Phi$ is the fixed part and $|F|$ the movable part. 
The general curve in $|C|$ is not hyperelliptic, 
hence $|F|$ is not composed with a pencil, 
so its general curve is irreducible. 
Moreover $\dim (|F|)=2$, $C\cdot F=4$, $C\cdot \Phi=0$ and $C\cdot K_{\tilde X}=1$.  

Consider  the exact sequence 
\begin{equation}\label{eq:sec}
0\longrightarrow \calO_{\tilde X}(\Phi-K_{\tilde X})\longrightarrow \calO_{\tilde X}(C)\longrightarrow \calO_{F}(C)\longrightarrow 0.
\end{equation}
We notice that the restriction map 
\begin{equation}\label{eq:seca}
r: H^0(\tilde X, \calO_{\tilde X}(C))\longrightarrow H^0( F, \calO_{F}(C))
\end{equation}
is injective. Indeed $h^0(\tilde X, \calO_{\tilde X}(\Phi-K_{\tilde X}))=0$ 
because $C\cdot (\Phi-K_{\tilde X})=-1$, and $C$ is nef. 

So $|C|$ cuts out on the general curve $F\in |F|$ a base point free $g^2_4$. 
This implies that either the general curve $F\in |F|$ has genus $p\leq 2$ 
or that it is hyperelliptic of genus $p\geq 3$. 

We claim that  the general curve $F\in |F|$ cannot be hyperelliptic of genus $p\geq 3$. Indeed, if this is the case, the $g^2_4$ cut out by $|C|$ on $F$ 
is composed with the $g^1_2$. 
Let $x$ be a general point of $\tilde X$. 
Consider the pencil $\calF_x$ of curves in $|F|$ containing $x$. 
If $F\in \calF_x$ is a general curve,  
let $x'$ be the conjugate of $x$ in the $g^1_2$ on $F$. 
Then $x'$ does not depend on $F$, 
otherwise it would describe a curve $\mathfrak F_x$ 
contained in all curves of $|C|$ containing $x$, 
which is not possible. 
So $x'$ is, as is $x$, a base point of $\calF_x$, 
and all curves in $|C|$ containing $x$ also contain $x'$. 
This implies that there is an involution $\iota: \tilde X \dasharrow \tilde X$ 
that maps $x$ to $x'$. 
Let $g: \tilde X\dasharrow Y$ be the quotient by this involution. 
By the above argument we see that the degree $3$ map $\tilde X\longrightarrow \bbP^2$, should factor through $g$, 
which is a contradiction. 
This proves our claim,  and we must have that 
the general curve $F\in |F|$ has genus $p\leq 2$.

The classification depends on the various values that $p$ can have.

Suppose first that  the general curve $F\in |F|$ has genus 2.  
We remark that the map $r$ is an isomorphism 
and $h^1( F, \calO_{F}(C))=0$. 
This implies that $h^1(\tilde X, \calO_{\tilde X}(\Phi-K_{\tilde X}))=1$ 
and $h^2(\tilde X, \calO_{\tilde X}(\Phi-K_{\tilde X}))=0$ 
(because $H^2(\tilde X, \calO_{\tilde X}(C))=0$);
hence $2K_{\tilde X}-\Phi$ is not effective 
and $\chi(\calO_{\tilde X}(\Phi-K_{\tilde X}))=-1$. 

We have $(K_{\tilde X}+F)\cdot F=2$, hence $K_{\tilde X}\cdot F=2-F^2$. 
Moreover
$$
1+K_{\tilde X}^2=K_{\tilde X}\cdot (C+K_{\tilde X})=K_{\tilde X}\cdot F+K_{\tilde X}\cdot \Phi=2-F^2+K_{\tilde X}\cdot \Phi,
$$
so that 
$$K_{\tilde X}^2=1-F^2+K_{\tilde X}\cdot \Phi.$$
In addition 
$$
4+F\cdot K_{\tilde X}=F\cdot (K_{\tilde X}+C)=F^2+F\cdot \Phi\geq F^2,
$$
hence
$$
6=4+F\cdot K_{\tilde X}+F^2\geq 2F^2
$$
so that $F^2\leq 3$, and if the equality holds one has $F\cdot \Phi=0$.

By the exact sequence 
$$
0\longrightarrow \calO_{\tilde X}\longrightarrow \calO_{\tilde X}(F)
\longrightarrow \calO_{F}(F)\longrightarrow 0
$$
since $h^1( \tilde X, \calO_{\tilde X})=0$ and $h^0( \tilde X, \calO_{\tilde X}(F))=3$, 
we have $h^0(F, \calO_{F}(F))=2$. 
This implies that $F^2\geq 2$ because 
if $\calF$ is a line bundle of degree at most $1$ on $F$, 
one has $h^0( F, \calF)\leq 1$. 
 
Also
$$
6=F\cdot C+F\cdot (F+K_{\tilde X})=F^2+F\cdot (K_{\tilde X}+C)=2F^2+F\cdot \Phi
$$ 
so that $F\cdot \Phi\in \{0, 2\}$. 

Finally, we have
$$
5+K^2_{\tilde X}=(C+K_{\tilde X})^2=F^2+\Phi^2+2F\cdot \Phi.
$$

Now we claim that $\Phi=0$. 
First we notice that if there is an irreducible curve $A$ contained in $\Phi$ 
such that $K_{\tilde X}\cdot A<0$, 
then, since $A^2<0$, the curve $A$ is a $(-1)$--curve such that $C\cdot A=0$, 
which is impossible. 
So, if $\Phi$ is non--zero, 
then for any curve $A$ in $\Phi$ 
one has $K_{\tilde X}\cdot A\geq 0$ and therefore $K_{\tilde X}\cdot \Phi\geq 0$. 

Now, suppose that $F\cdot \Phi=0$ which then gives $F^2=3$. Then
$$
5+K^2_{\tilde X}=3+\Phi^2\leq 2, \quad {\rm i.e.,}\quad K^2_{\tilde X}\leq -3.
$$

Since $\chi(\calO_{\tilde X}(\Phi-K_{\tilde X}))=-1$, we have
$$
-4=\Phi^2-3K_{\tilde X}\cdot \Phi+2K_{\tilde X}^2\leq \Phi^2-3K_{\tilde X}\cdot \Phi-6,
$$
hence
$$
2+3K_{\tilde X}\cdot \Phi \leq \Phi^2\leq -1
$$
thus $K_{\tilde X}\cdot \Phi < 0$, which is a contradiction.

Assume next that $F\cdot \Phi=2$, so that $F^2=2$. 
Then $|F|$ is base point free 
and $|F|$ cuts out on the general curve $F\in |F|$ the canonical $g^1_2$. 
Then $\varphi_{|F|}: \tilde X\longrightarrow \bbP^2$ 
is a double cover with a branch curve $D$ of degree $6$. 
Since $p_g=0$, then $D$ cannot have only irrelevant singularities, 
hence $D$ must have either a point of multiplicity (at least) $4$ 
or a triple point with an infinitely near triple point 
(a so called $[3,3]$--point). 
Hence $\tilde X$ is rational 
(see \cite [\S 8.4]{C}, and  $-K_{\tilde X}$ is effective. 
But this gives a contradiction, since $C\cdot K_{\tilde X}=1$ and $C$ is nef. 

So we have proved that, 
if the general curve $F\in |F|$ has genus $2$, $\Phi=0$ and therefore $F^2=3$. 
The system $|F|$ can have at most one simple base point 
and, by Bertini's theorem, the general curve in $|F|$ is smooth 
and, as we know, irreducible. 
Moreover $h^1(\tilde X, \calO_{\tilde X}(F))=0$ 
and since $K _{\tilde X}\cdot F=-1$, 
the Kodaira dimension of $\tilde X$ is $-\infty$ and therefore $\tilde X$ is rational.
Then, by \cite [Prop. 10.10, Rem. 10.11]{CaCi}, 
up to a Cremona transformation $|F|$ can be identified with one of these systems:
$|6;2^8,1|$ and $|4;2, 1^9|$. 
Since $F\sim K_{\tilde X}+C$, 
these two cases give rise to cases (i) and (ii) of the statement of the theorem. 

To finish we need to analyze the two cases 
in which the curves of $|F|$ have  genus one and genus zero.

First we assume that the general curve $F\in |F|$ has genus 1. 
Then  $F\cdot K_{\tilde X}=-F^2<0$, 
which implies that $\tilde X$ is rational. 
Then, up to a Cremona transformation, 
we have that $|F|$ can be identified with the linear system $|3;1^7|$; 
hence $F^2=2$, so that $F\cdot K_{\tilde X}=-2$. Thus
$$
1+K_{\tilde X}^2=K_{\tilde X}\cdot (C+K_{\tilde X})
=K_{\tilde X}\cdot F+K_{\tilde X}\cdot \Phi=K_{\tilde X}\cdot \Phi-2, 
\quad {\rm i.e.,}\quad K_{\tilde X}\cdot \Phi=K_{\tilde X}^2+3.
$$
Going back to the sequence \eqref{eq:sec}, 
recall that the map $r$ in \eqref {eq:seca} is injective 
because $\Phi-K_{\tilde X}$ is not effective. 
In this case we have $h^0(F, \calO_F(C))=4$ and  $h^1( F, \calO_{F}(C))=0$. 
Since $h^1(\tilde X, \calO_{\tilde X}(C))=1$ and $h^2(\tilde X, \calO_{\tilde X}(C))=0$, 
we have that $h^1(\tilde X, \calO_{\tilde X}(\Phi-K_{\tilde X}))=2$ 
and $h^2(\tilde X, \calO_{\tilde X}(\Phi-K_{\tilde X}))=0$, 
so that  $\chi(\calO_{\tilde X}(\Phi-K_{\tilde X}))=-2$. 
Hence we have
$$
-6=\Phi^2-3K_{\tilde X}\cdot \Phi+2K_{\tilde X}^2= \Phi^2-K_{\tilde X}\cdot \Phi-6,
$$
giving
\begin{equation}\label{eq:kip}
K_{\tilde X}\cdot \Phi= \Phi^2. 
\end{equation}
Since, as we saw above, $K_{\tilde X}\cdot \Phi<0$ is not possible, 
but if $\Phi \neq 0$ then $\Phi^2<0$; 
hence we have $\Phi=0$ and $K_{\tilde X}^2=-3$. 
Then we are in case (iii).

Finally, suppose that the general curve in $|F|$ is rational. 
Then, up to a Cremona transformation, 
we have that $|F|$ can be identified with the linear system of  lines in the plane, 
so that  $F^2=1$ and $K_{\tilde X}\cdot F=-3$. 
In this case we have
$$
1+K_{\tilde X}^2=K_{\tilde X}\cdot (C+K_{\tilde X})
=K_{\tilde X}\cdot F+K_{\tilde X}\cdot \Phi
=K_{\tilde X}\cdot \Phi-3, \quad {\rm i.e.,}\quad K_{\tilde X}\cdot \Phi=K_{\tilde X}^2+4,
$$
and $\chi(\calO_{\tilde X}(\Phi-K_{\tilde X}))=-3$. 
An easy computation as before shows that \eqref {eq:kip} still holds 
and as above this implies that $\Phi=0$. 
Then we are in case (iv). 
\end{proof}

\begin{remark}\label{rem:skip} 
At the beginning of the proof of Proposition \ref {prop:plop} 
we saw that $h^1(\tilde X, \calO_{\tilde X}(C))=1$. 
This implies that for the linear systems 
listed in the statement of Proposition \ref {prop:plop} 
the base points cannot be general, 
since otherwise the systems in question would have dimension $1$ rather than $2$. However there are points in special position 
so that these linear systems do have dimension $2$.  

Let us look more closely at case (i).  
In this case the map $\tilde X \to \bbP^2$ 
cannot be the finite triple cover map: $\tilde X$ is different from $X$.  
Suppose to the contrary that this is the triple cover map.  
Then by Enriques' Lemma (\cite{I}, page 984), 
we deduce that the geometric genus of the branch curve is $16$.
However if this is the triple cover map, 
then the branch curve has in general only cusps, 
and the number of cusps is computed by Enriques to be $24$ 
(see again \cite{I}, same page).  
However a curve of degree $10$ with $24$ cusps has genus $12$, not $16$.

We conclude that the branch curve must have further singularities besides cusps, 
and there must be some curve $D$ contracted by the linear system.
For any such irreducible curve, one concludes readily 
(using the restriction sequence to $D$) 
that $D$ has genus at most one.  
In fact we can construct such examples as follows.

Suppose the linear system $|9;3^8,2, 1^2|$ 
has the first $10$ base points on an irreducible cubic curve. 
Then the proper transform of the cubic on the blow--up 
has zero intersection with the system $|9;3^8,2, 1^2|$ 
and it splits off the system with one condition. 
The residual system is $|6;2^8,1^2|$ that is a pencil. 
This proves that in this case $|9;3^8,2, 1^2|$ has dimension $2$ as desired. 
The cubic is contracted to an elliptic singularity on $X$. 

Notice that in this case the linear system $|9;3^8,2, 1|$ has dimension $3$ 
and it determines a rational map of $\bbP^2$ to $\bbP^3$ 
whose image is a quartic with an isolated genus $1$ singularity 
corresponding to the contraction of the cubic curve. 
This is a so--called \emph{Noether quartic} (see \cite [p. 204]{Co}). 
The triple plane is determined by the projection of the Noether quartic to $\bbP^2$ 
from a smooth point on it. 

In case (ii) the situation is quite similar. 
With the same considerations as above, 
one sees that there must be some curve $D$ on $\tilde X$ 
contracted by the linear system. 
Again any such a curve has genus at most one.  
In fact we can construct such examples as follows.

Suppose the linear system $|7;3,2^9, 1|$ 
has the first $10$ base points on an irreducible cubic curve. 
Then the proper transform of the cubic on the blow--up 
has zero intersection with the system $|7;3,2^9, 1|$ 
and it splits off the system with one condition. 
The residual system is $|4;2,1^{10}|$ which is a pencil. 
This proves that in this case $|7;3,2^9, 1|$ has dimension $2$ as desired. 
The cubic is contracted to an elliptic singularity on $X$. 

Again in this case the linear system $|7;3,2^9|$ has dimension $3$ 
and it determines a rational map of $\bbP^2$ to $\bbP^3$ 
whose image is a quartic with an isolated genus $1$ singularity 
(a so--called \emph{tacnode}) 
corresponding to the contraction of the cubic curve. 
The triple plane is determined by the projection of this quartic to $\bbP^2$ 
from a smooth point on it. 

Also in case (iii) the situation is similar. 
With the same considerations as above, 
one sees that there must be some curve $D$ on $\tilde X$ 
contracted by the linear system. 
Again any such a curve has genus at most one.  
We can construct such examples as follows.

Suppose the linear system $|6;2^7, 1^5|$ 
has the first $11$ base points on an irreducible cubic curve. 
Then the proper transform of the cubic on the blow--up 
has zero intersection with the system $|6;2^7, 1^5|$ 
and it splits off the system with one condition. 
The residual system is $|3;1^8|$ which is a pencil. 
This proves that in this case $|6;2^7, 1^5|$ has dimension $2$ as desired. 
The cubic is contracted to an elliptic singularity on $X$. 

Again in this case the linear system $|6;2^7, 1^4|$ has dimension $3$ 
and it determines a rational map of $\bbP^2$ to $\bbP^3$ 
whose image is a quartic with an isolated genus $1$ singularity 
corresponding to the contraction of the cubic curve. 
The triple plane is determined by the projection of this quartic to $\bbP^2$ 
from a smooth point on it. 

A second example with genus zero may be obtained as follows.  
Take a nodal cubic curve $\bar D$, 
and choose $8$ additional simple points on it.
Fix six of these, and choose a general pencil of cubics through those six.
These will meet in three additional points.
Now take the linear system of sextics
which are double at the node of $\bar D$, also at the chosen six points,
and pass simply through the two additional points of $\bar D$ 
and the three further base points of the cubic pencil.
The proper transform $D$ of the cubic is contracted by the linear system, 
and is a rational curve of self-intersection $-3$ on $\tilde X$.

Similarly, another example with a contracted genus one curve may be given as follows, 
with the other case (iv).  
Suppose that the linear system $|4;1^{13}|$ 
has its first $12$ base points on a smooth cubic curve. 
Then the proper transform of the cubic on the blow--up 
has zero intersection with the system $|4;1^{13}|$ and it splits off the system with one condition. 
The residual system is $|1;0^{12},1|$ which is a pencil. 
This proves that in this case $|4;1^{13}|$ has dimension $2$ as desired.
It is immediate that this triple plane can be obtained by projecting to $\bbP^2$ 
a quartic with an isolated triple point from a smooth point on it.  

One more example is to take four points on a line, 
and $9$ points which are the intersection of two cubics.  
The system of quartics through these $13$ points give an example where the line is contracted.

We do see constructions without contracted curves in case (iv).  
Take a smooth quartic curve $C$ and a line $\ell$ meeting $C$ in four points.  
Choose one of these points, 
and consider another quartic $C'$ passing through the other three points.  
This quartic intersects $C$ in $13$ additional points,
and the linear system of quartics through these $13$ points 
gives an example of the case (iv) without having any contracted curves. 

In case (iv) of Proposition \ref {prop:plop}, 
where there is no contraction,
the Tschirnhausen vector bundle is $\Omega^1_{\bbP^2}(-1)$ 
(see \cite [p. 1158]{M}).  
Since this bundle is not split, 
it is not easy to give explicit equations for the structure constants of the triple cover map.
Since this bundle has $c_2=7$, the branch curve of degree $10$ 
has in general $21$ cusps (see \cite[Lemma 10.1]{M}).

\end{remark}

\begin{example}\label{ex:gul} 
It is enlightening to look at the two cases (iii) and (iv) in Proposition \ref {prop:plop} 
and at the content of Remark \ref {rem:skip} 
for an even more geometric point of view. 

First of all one can consider quintic surfaces $S$ 
with three non--planar double lines concurring at a triple point. 
These surfaces are rational, and it is proven in  \cite [\S 33]{Co} that 
they are the image in $\bbP^3$ of the plane blown up at $9$ points 
via a linear system of the type $|6;2^7, 1^3|$. 
Now take a general point $x$ on one of the double lines $r$ 
and make the projection of $S$ from $x$. 
This determines a triple plane that falls in case (i) of Proposition \ref {prop:plop}. 
This projection contracts $r$ to a point 
according to what we observed in Remark \ref {rem:skip}. 

One can also consider quintic surfaces $S$ 
with a rational normal cubic curve $\Gamma$ of double points. 
These surfaces are rational, and it is proven in  \cite [\S 33]{Co} that 
they are the image in $\bbP^3$ of the plane blown up at $11$ points 
via a linear system of the type $|4;1^{11}|$. 
Now take a general point $x\in \Gamma$ and make the projection of $S$ from $x$. 
This determines a triple plane that falls in case (iv) of Proposition \ref {prop:plop}. 
In this projection in general there are no contractions. 
However, the double curve $\Gamma$ can specialize to a conic plus a line $r$ 
and if we take $x\in r$ general and project down from $x$, 
the line $r$ will be contracted to a point. 
Again this agrees with the content of Remark \ref {rem:skip}. 
\end{example}

\section{Branch curve of degree $12$ and higher}\label{sec:12}

As we saw, if the branch curve has degree $10$ or lower, 
a triple plane has Kodaira dimension $-\infty$, 
unless the branch curve has degree $10$,
in which case it could be either properly elliptic 
or a $K3$ quartic surface blown up at a smooth point 
(and we are not sure that the elliptic case exists). 
If the branch curve has degree $12$ or higher, 
the Kodaira dimension can get higher as well, 
and it becomes difficult to come up with a classification. 
We give here some examples, in particular with branch curves of degree $12$.

\begin{example}\label{ex:gt} 
Consider a surface $X$ of general type with $p_g=K^2=3$ 
and with base point free canonical system $|K|$.  
If $f: X\longrightarrow \bbP^2$ is the canonical map, 
this is a triple plane and the Tschirnhausen vector bundle is 
$\calO_{\bbP^2}(-2)\oplus \calO_{\bbP^2}(-4)$ (see \cite [Cor. 10.4]{M}). 
The branch curve has degree $12$ and in general has $24$ cusps 
(see \cite [Lemma 10.1]{M}). 
These surfaces have been studied in \cite [Sect. 2]{Ho}.
\end{example} 

\begin{example}\label{ex:ell}
 If the Tschirnhausen vector bundle is $\calO_{\bbP^2}(-3)\oplus \calO_{\bbP^2}(-3)$  then we have a triple plane $\pi: X\longrightarrow \bbP^2$ 
which is an elliptic surface over $\bbP^1$ 
(the elliptic structure being given by the canonical map, so $p_g=2$, $K^2=0$) 
and the triple covering is defined by a linear system 
of genus $4$ trisections of the elliptic structure (see \cite [Table 10.5]{M}). 
The branch curve has degree $12$ and in general has $27$ cusps 
(see \cite [Lemma 10.1]{M}).
\end{example} 

\begin{example}\label{ex:ell2} 
Consider a sextic surface $S$ in $\bbP^3$ 
which is double along the edges of a tetrahedron and has no other singularities. 
Its normalization $S'$ is an Enriques surface. 
The projection of $S$ to a plane from a vertex of the tetrahedron 
determines a triple plane $\pi: X\longrightarrow \bbP^2$ 
where $X$ is the blow--up of $S'$ at the three points 
that are mapped to the vertex of the tetrahedron. 
The branch curve of this triple plane has degree $12$ 
because the curves $C\in \calL$ have genus $4$. 
Taking into account  \cite [Cor. 10.4]{M}, 
we think that the Tschirnahusen vector bundle does not split in this case. 
\end{example}

\begin{example}\label{ex:K3} 
Let $S$ be a smooth surface of degree $6$ in $\bbP^4$ 
that is the complete intersection of a quadric and a cubic hypersufaces. 
This is a K3 surface. 
There are proper trisecant lines to $S$, 
precisely all lines in the quadric and not contained in $S$. 
One can project $S$ down to $\bbP^2$ from any such proper trisecant line. 
This determines a triple plane that has a branch curve of degree $12$. 
The Tschirnausen bundle here is not split 
but it is somehow \emph{very close} to being split, 
i.e., it is a quotient of a split rank three vector bundle by a line bundle 
(see \cite [Table 1]{IPR}).

Of course one can consider special cases in which $S$ becomes singular, 
in particular it becomes rational (e.g., if $S$ acquires a triple point). 
\end{example}


\begin{example}\label{ex:ello}  
Infinitely many examples can be obtained by considering the exceptional examples 
(see Section \ref {ssec:exc}) 
and the triple planes with a pencil composed with the triple plane map 
(see Section \ref {ssec:comp}  and \eqref {eq:comp}).  
In particular, if in \eqref {eq:comp} we have that $\Gamma = \bbP^1$ 
and $f: \bbP^2\dasharrow \bbP^1$ is given by a pencil of curves of degree $d$, 
the branch curve consists of $4d$ curves in the pencil. 
For $d=3$ we have a branch curve of degree $12$ and $X$ is an elliptic surface. 
\end{example}

\section{Appendix: Triple Cover Formulae}\label{triplecoverformulae}
In this appendix we review how to obtain the structure parameters $a,b,c,d$ 
(in the notation of \cite{M})
from the coefficients $e_j$ of a monic cubic equation
\[
X^3 + e_2 X^2 + e_1 X + e_0 = 0
\;\;\;\text{ or }\;\;\; X^3 = -e_0 -e_1 X -e_2 X^2.
\]
The algebra is generated by $1,X,X^2$ as a module,
but the structure constants that we seek require us to find
the trace zero generators.
Note that 
\begin{align*}
	X^4 &= -e_0 X -e_1 X^2 - e_2 X^3 \\
	&= -e_0 X -e_1 X^2 -e_2(-e_0-e_1 X-e_2 X^2) \\
	&= e_0e_2 + (e_1e_2-e_0)X + (e_2^2-e_1)X^2   
\end{align*}
Since
\[
{\rm Tr}(X) = 
{\rm Tr}\begin{pmatrix} 0 & 0 & -e_0\\ 1 & 0 & -e_1 \\ 0 & 1 & -e_2 \end{pmatrix}
= -e_2
\]
and
\[
{\rm Tr}(X^2) = 
{\rm Tr}\begin{pmatrix} 0 & -e_0 & e_0e_2 \\ 0 & -e_1 & (e_1e_2-e_0) \\ 1 & -e_2 & (e_2^2-e_1) \end{pmatrix}
=e_2^2-2e_1
\]
we therefore have that
\[
z = X+e_2/3 \;\;\;\text{ and }\;\;\; w = X^2 -(e_2^2-2e_1)/3
\]
are the trace zero generators of the module.
Using the notation of \cite{M},
we seek the structure constants $a,b,c,d$ for multiplication in the rank three algebra
in terms of the generators $1$, $z$ and $w$, which are
\begin{align} \label{triplecoverstructureconstants}
	z^2 &= 2A + az+bw \nonumber\\
	zw &= -B -dz-aw\\
	w^2 &= 2C +cz+dw \nonumber
\end{align}
where
\[
A = a^2-bd,\;\;\; B=ad-bc,\;\;\; C=d^2-ac.
\]
We then compute these as:
\begin{align*}
	z^2 &= (X+e_2/3)^2 = X^2 + (2e_2/3)X+(e_2^2/9) \\
	&= (w+(e_2^2-2e_1)/3) + (2e_2/3)(z-e_2/3) +(e_2^2/9) \\
	&= (2e_2^2/9 -2e_1/3) + (2e_2/3) z + w, \\
	zw &= (X+e_2/3)(X^2 -(e_2^2-2e_1)/3) \\
	&=  X^3 + (e_2/3)X^2 -((e_2^2-2e_1)/3)X - (e_2^3-2e_1e_2)/9 \\
	&= -e_2X^2 - e_1X - e_0 + (e_2/3)X^2 -((e_2^2-2e_1)/3)X - (e_2^3-2e_1e_2)/9 \\
	&= (-2e_2/3)X^2 + ((-e_2^2-e_1)/3)X - e_0 - (e_2^3-2e_1e_2)/9 \\
	&= (-2e_2/3)(w+(e_2^2-2e_1)/3) + ((-e_2^2-e_1)/3)(z-e_2/3) - e_0 - (e_2^3-2e_1e_2)/9 \\
	&= (-e_0 -2e_2^3/9 +7e_1e_2/9) + ((-e_2^2-e_1)/3) z + (-2e_2/3) w \\
	w^2 &= (X^2 -(e_2^2-2e_1)/3)^2 \\
	&= X^4 -2(e_2^2-2e_1)/3) X^2 + (e_2^2-2e_1)^2/9 \\
	&= (e_2^2-e_1)X^2 + (e_1e_2-e_0)X + e_0e_2 
	-2(e_2^2-2e_1)/3) X^2 + (e_2^2-2e_1)^2/9 \\
	&= ((e_2^2+e_1)/3) X^2 + (e_1e_2-e_0)X + e_0e_2 + (e_2^2-2e_1)^2/9 \\
	&= ((e_2^2+e_1)/3)(w+(e_2^2-2e_1)/3) + (e_1e_2-e_0) (z-e_2/3) + e_0e_2 + (e_2^2-2e_1)^2/9 \\
	&= (2e_2^4-8e_1e_2^2+2e_1^2+12e_0e_2)/9 + (e_1e_2-e_0) z +  ((e_2^2+e_1)/3) w
\end{align*}
Therefore:
\begin{proposition}\label{TCstructure}
If a rank three algebra is defined by an equation
\[
X^3 + e_2 X^2 + e_1 X + e_0 = 0,
\]
then the structure constants for the algebra
using trace zero generators $z = X+e_2/3$ and $w = X^2 -(e_2^2-2e_1)/3$ 
as in \eqref{triplecoverstructureconstants}
are given by
\begin{align*}
	a &= 2e_2/3, \;\;\; b = 1, \;\;\; c = e_1e_2-e_0, \;\;\; d = (e_2^2+e_1)/3 \\
	A &= e_2^2/9 -e_1/3 \\
	B &= e_0 + 2e_2^3/9 - 7e_1e_2/9 \\
	C &= (e_2^4 - 4e_1e_2^2+ e_1^2+ 6e_0e_2)/9
\end{align*}
using the notation of \cite{M}.
\end{proposition}

It is elementary but onerous to check that we have the necessary relations among these seven quantities ($A = a^2-bd$, $B=ad-bc$, and $C=d^2-ac$).


\begin{thebibliography}{}

\bibitem {B1} J. Bronowski, \emph{On triple planes. I}, J. London Math. Soc., {\bf 12}, (1937), 212--216.

\bibitem {B2} J. Bronowski, \emph{On triple planes. II}, J. London Math. Soc., {\bf 17}, (1942), 24--31.

\bibitem {B3} J. Bronowski, \emph{On triple planes. III}, J. London Math. Soc., {\bf 17}, (1942), 80--87.

\bibitem {C} A. Calabri, \emph{Rivestimenti del piano, Sulla razionalit\`a dei piani doppi e tripli ciclici}, Centro Studi Enriques Ed. Plus, Pisa University Press, 2006.

\bibitem {CaCi} A. Calabri, C. Ciliberto, \emph{Birational classification of curves on rational surfaces}, Nagoya Math. J., {\bf 199}, (2010), 43--93.

\bibitem {CC} F. Catanese, C. Ciliberto, \emph {Symmetric products of elliptic curves and surfaces of general type with $p_g = q = 1$}, J. of Alg. Geom., {\bf 2} (1993), 389--411.

\bibitem {CCML} F. Catanese, C. Ciliberto, M. Mendes Lopes, \emph{On the classification of irregular surfaces of general type with non birational bicanonical map}, Transactions in Math. {\bf 350} (1) (1998), 275--308 

\bibitem {C} C. Ciliberto, \emph{Classification of Complex Algebraic Surfaces},  Series of Lectures in Math., EMS,  2020. 

\bibitem {CFM} 
C. Ciliberto, P. Francia, M. Mendes Lopes, 
{\em Remarks on the bicanonical map for surfaces of general type}, 
Math. Z., {\bf 224} (1997), 137--166.

\bibitem {CE} G. Casnati, T. Ekedahl, \emph{Covers of algebraic varieties. I. A general structure theorem, covers of degree 3, 4 and Enriques surfaces}, J. Algebraic Geom., {\bf 5}, (1996), no. 3, 439--460.

\bibitem {Co} F. Conforto, \emph{Le superficie razionali}, Zanichelli, Bologna, 1939.

\bibitem {Du1} P. d. Val, \emph{ On triple planes having branch curves of order not greater than twelve}, J. London Math. Soc., {\bf 8} (1933), no. 3, 199--206. 

\bibitem {Du2}  P. du Val, \emph{ On triple planes whose branch curves are of order fourteen}, Proc. London Math. Soc. (2) {\bf 39} (1935), no. 1, 68--81. 

\bibitem {FPV} D. Faenzi, F. Polizzi, J Vall\`es, \emph{Triple planes with $p_g= q= 0$},  Transactions of the American Mathematical Society, {\bf 371}.1, (2019), 589--639.

\bibitem {F} A. Franchetta, \emph{Sulle curve riducibili appartenenti ad una superficie algebrica}, Rend. di mat. e appl., Ser. V, {\bf 8}, 3-4, (1949), 1-21.

\bibitem {har} R. Hartshorne, \emph{Algebraic Geometry}, Graduate Texts in Math., Springer Verlag, 1977.

\bibitem {Ho}  E. Horikawa, \emph{Algebraic Surfaces of General Type With Small $c^2_1$, II}, Invent. Math., {\bf 37}, (1976), 121--155.

\bibitem {IPR} N. Istrati, P. Pokora, S. Rollenske, \emph{Special triple covers of algebraic surfaces}, Documenta Math., {\bf 27}, (2022), 2301--2332.

\bibitem{I} B. Iversen, \emph{Numerical Invariants and Multiple Planes},
American Journal of Mathematics, {\bf 92}, No. 4 (Oct., 1970), 968--996.

\bibitem  {ML} M.~Mendes Lopes, 
{\em Adjoint systems on surfaces}, 
Boll. Un. Mat. Ital., A (7)  {\bf 10}  (1996),  no. 1, 169--179.

\bibitem {M} R. Miranda, \emph{Triple Covers in Algebraic Geometry}, American Journal of Mathematics, {\bf 107} (5),  (1985), 
1123--1158.

\bibitem {P0} G. Pompilj, \emph{Sulla rappresentazione algebrica dei piani tripli}, Rend. Sem. Mat. Roma, (4) {\bf 3}, (1939), 109--132. 

\bibitem {P01} G. Pompilj, \emph{Osservazioni sui piani tripli}, Ist. Lombardo Sci. Lett. Cl. Sci. Mat. Nat. Rend. (3) {\bf 5} (74), (1941), 263--279. 

\bibitem {P1} G. Pompilj, \emph{Sui piani tripli con quartica di diramazione}, Ann. di Mat. Pura e Appl.,  {\bf 24} (4), (1945), 65--117.

\bibitem {P} G. Pompilj, \emph{Sui piani tripli con un fascio irrazionale}, Rend. della R. Acc. Nazionale dei Lincei, {\bf 1} (8), (1946), 306--313.

\bibitem {P2} G. Pompilj, \emph{  Sui piani tripli birazionalmente identici}, Atti Accad. Naz. Lincei. Rend. Cl. Sci. Fis. Mat. Nat. (8) {\bf 1} (1946), 318--322.

\bibitem {T} S.-L. Tan, \emph{Triple covers on smooth algebraic varieties},  AMS IP Studieds in Advanced Math., {\bf 29} (2002), 143--164.

\end{thebibliography}
\end{document}